\documentclass[final,1p,times]{elsarticle}

\usepackage{amssymb}
\usepackage{amsthm}
\usepackage{amsmath,amssymb,amsopn,amsfonts,mathrsfs,amsbsy,amscd}
\usepackage{graphicx}

\newcommand{\dev}{\mathrm{div}}

\newcommand{\br}{[\;,\;]}
\newcommand{\too}{\longrightarrow}
\newcommand{\om}{\omega}
\newcommand{\esp}{\quad\mbox{and}\quad}

\newcommand{\G}{{\mathfrak{g}}}

\newcommand{\Li}{{\cal L}}
\newcommand{\B}{{\cal B}}

\newcommand{\D}{{\cal D}}

\newcommand{\di}{\displaystyle}
\newcommand{\Om}{\Omega}
\newcommand{\na}{\nabla}
\newcommand{\wi}{\widetilde}
\newcommand{\al}{\alpha}
\newcommand{\be}{\beta}

\newcommand{\ga}{\gamma}
\newcommand{\Ga}{\Gamma}
\newcommand{\e}{\epsilon}

\font\bb=msbm10

\def\B{\hbox{\bb B}}
\def\R{\hbox{\bb R}}

\def\C{\hbox{\bb C}}
\newtheorem{theorem}{Theorem}[section]
\newtheorem{lemma}[theorem]{Lemma}

\newtheorem{definition}[theorem]{Definition}
\newtheorem{example}[theorem]{Example}

\newtheorem{proposition}[theorem]{Proposition}
\newtheorem{remark}[theorem]{Remark}
\newtheorem{corollary}[theorem]{Corollary}

\numberwithin{equation}{section}
\newcommand{\dreqno}{\let\veqno\eqno}

\begin{document}
	
	\begin{frontmatter}
		
		
		
		
		\title{  Contravariant Pseudo-Hessian manifolds and their associated Poisson structures}
		
		\author[label1]{Abdelhak Abouqateb}
		\address[label1]{Universit\'e Cadi-Ayyad\\
			Facult\'e des sciences et techniques\\
			BP 549 Marrakech Maroc\\e-mail: a.abouqateb@uca.ac.ma
		}
		\author[label2]{Mohamed Boucetta}
		\address[label2]{Universit\'e Cadi-Ayyad\\
			Facult\'e des sciences et techniques\\
			BP 549 Marrakech Maroc\\e-mail: m.boucetta@uca.ac.ma
		}
		
		\author[label3]{Charif Bourzik}
		\address[label3]{Universit\'e Cadi-Ayyad\\
			Facult\'e des sciences et techniques\\
			BP 549 Marrakech Maroc\\e-mail: bourzikcharif@gmail.com
		}
		
		
		
		\begin{abstract} A contravariant pseudo-Hessian manifold  is a manifold $M$ endowed with a pair $(\nabla,h)$ where $\nabla$ is a flat connection and $h$ is a symmetric bivector field satisfying a contravariant Codazzi equation. When $h$ is invertible we recover the known notion of pseudo-Hessian manifold.
		Contravariant pseudo-Hessian manifolds have properties similar to Poisson manifolds and, in fact,  to any contravariant pseudo-Hessian manifold $(M,\na,h)$ we associate naturally a Poisson tensor on $TM$. We investigate these properties and we study in details many classes of such structures in order to highlight the richness of the geometry of these manifolds.

		\end{abstract}
		
		
	\end{frontmatter}
	
	

	
	
	
\section{Introduction}\label{section1}
A contravariant pseudo-Hessian manifold is an affine manifold $(M,\na)$ endowed with a symmetric bivector field $h$ such that, for any $\al,\be,\ga\in\Om^1(M)$,
\begin{equation}\label{gg}
(\nabla_{h_{\sharp}(\alpha)}h)(\beta,\gamma)=(\nabla_{h_{\sharp}(\beta)}h)(\alpha,\gamma),
\end{equation}where $h_\#:T^*M\too TM$ is the contraction. We will refer to \eqref{gg} as contravariant Godazzi equation. These manifolds where introduced in \cite{Be} as a generalization of pseudo-Hessian manifolds. Recall that a pseudo-Hessian manifold is an affine manifold $(M,\na)$ with a pseudo-Riemannian metric $g$ satisfying the Godazzi equation 
\begin{equation}\label{codazzi}
 \na_{X}g(Y,Z)=\na_{Y}g(X,Z),\quad  \end{equation}for any $X,Y,Z\in\Ga(TM)$.
The book \cite{Shi} is devoted to the study of Hessian manifolds which are pseudo-Hessian manifolds with a Riemannian metric. 

In this paper, we study contravariant pseudo-Hessian manifolds. The passage from pseudo-Hessian manifolds to contravariant pseudo-Hessian manifolds is similar to the passage from symplectic manifolds to Poisson manifolds and this similarity will guide our study. Let $(M,\na,h)$ be a contravariant pseudo-Hessian manifold. We will show that $T^*M$ has a Lie algebroid structure, $M$ has a singular  foliation whose leaves are pseudo-Hessian manifolds and $TM$ has a Poisson tensor whose symplectic leaves are pseudo-K\"ahlerian manifolds. We investigate an analog of Darboux-Weinstein's theorem and we show that it is not true in general but holds in some cases. We will study in details the correspondence which maps a contravariant pseudo-Hessian bivector field on $(M,\na)$ to a Poisson bivector field on $TM$.  We study affine, linear and quadratic contravariant pseudo-Hessian structures on  vector spaces and we show that an affine contravariant pseudo-Hessian structure on a vector space $V$ is equivalent to an associative commutative algebra product and a  2-cocycle on $V^*$. We study right invariant contravariant pseudo-Hessian structures on a Lie group $G$ and we show that $TG$ has a structure of Lie group (different from the one associated to  the adjoint action) for which the associated Poisson tensor is right invariant. We show that a right invariant contravariant pseudo-Hessian structure on a Lie group is equivalent to a $S$-matrix on the associated left symmetric algebra (see \cite{bai, Be2}) and we associate to any $S$-matrix on a left symmetric algebra $\G$ a solution of the classical Yang-Baxter equation on $\G\times\G$. Finally, we show that an action of a left symmetric algebra $\G$ on an affine manifold $(M,\na)$ transforms a $S$-matrix on $\G$ to a contravariant pseudo-Hessian bivector field on $(M,\na)$. Since the Lie algebra of affine vector fields of $(M,\na)$ has a natural structure of finite dimensional associative algebra, we have a mean to define contravariant pseudo-Hessian structures on any affine manifold. The paper contains many examples of contravariant pseudo-Hessian structures.

The paper is organized as follows. In Section \ref{section2}, we give the definition of a contravariant pseudo-Hessian manifold and we investigate its properties. In Section \ref{section3}, we study in details the Poisson structure of the tangent bundle of  a contravariant pseudo-Hessian manifold. Section \ref{section4} is devoted to the study of linear and affine contravariant pseudo-Hessian structures. Quadratic contravariant pseudo-Hessian structures will be studied in Section  \ref{section5}. In Section \ref{section6}, we study right invariant pseudo-Hessian structures on Lie groups.

\section{Contravariant pseudo-Hessian manifolds: definition and principal properties}
\label{section2}

\subsection{Definition of a contravariant pseudo-Hessian manifold}
 Recall that an affine manifold is a $n$-manifold   $M$ endowed with   a maximal atlas such that all transition functions are restrictions of elements of the affine group $\mathrm{Aff}(\R^n)$. This is equivalent to the existence on $M$ of a flat connection $\nabla$, i.e., torsionless and with vanishing curvature (see \cite{Shi} for more details). An affine coordinates system on an affine manifold $(M,\na)$ is a coordinates system $(x_{1},\ldots,x_{n})$ satisfying $\nabla\partial_{x_{i}}=0$ for any $i=1,\ldots,n$.

 Let $g$ be a pseudo-Riemannian metric on an affine manifold $(M,\nabla)$. The triple  $(M,\nabla,g)$ is called a {\it pseudo-Hessian manifold} if $g$ can be locally expressed in any affine coordinates system $(x_1,\ldots,x_n)$ as
 \[ g_{ij}=\frac{\partial^2\phi}{\partial x_i\partial x_j}. \]
 That  is equivalent to $g$ satisfying the {\it Codazzi equation} \eqref{codazzi}.
  When $g$ is Riemannian, we call $(M,\nabla,g)$ a Hessian manifold. The geometry of Hessian manifolds was studied intensively in \cite{Shi}.
 
 We consider now a more general situation.
 
 \begin{definition}[\cite{Be}] Let $h$ be a symmetric bivector  field on an affine manifold $(M,\na)$ and $h_{\sharp}:T^{*}M\rightarrow TM$ the associated contraction  given by $\beta(h_{\sharp}(\alpha))=h(\alpha,\beta)$. The triple $(M,\nabla,h)$ is called
 a {\it contravariant pseudo-Hessian manifold} if $h$ satisfies the \textit{contravariant Codazzi equation}
 \begin{equation}\label{codazzib}
 (\nabla_{h_{\sharp}(\alpha)}h)(\beta,\gamma)=(\nabla_{h_{\sharp}(\beta)}h)(\alpha,\gamma),\quad
 \end{equation}for any $\al,\be,\ga\in\Om^1(M)$. We call such $h$ a  pseudo-Hessian bivector field.
  
\end{definition}

One can see easily that if $(M,\na,g)$ is a pseudo-Hessian manifold then $(M,\na,g^{-1})$ is a contravariant pseudo-Hessian manifold.

The following proposition is obvious and gives the local expression of the equation \eqref{codazzib} in affine charts.
\begin{proposition}\label{loca} Let $(M,\na,h)$ be an affine manifold endowed with a symmetric bivector field. Then $h$ satisfies \eqref{codazzib} if and only if, for any $m\in M$, there exists an affine coordinates system $(x_1,\ldots,x_n)$ around $m$ such that for any $1\leq i<j\leq n$ and any $k=1,\ldots,n$
	\begin{equation}\label{eq21}
	\sum_{l=1}^{n}\left[h_{il}\partial_{x_{l}}(h_{jk})-h_{jl}\partial_{x_{l}}(h_{ik})\right]=0,
	\end{equation}where $h_{ij}=h(dx_i,dx_j)$.

\end{proposition}

\begin{example}\begin{enumerate}
		\item Take $M=\R^n$ endowed with its canonical affine structure and consider
		\[ h=\sum_{i=1}^n f_i(x_i)\partial_{x_i}\otimes \partial_{x_i}, \]where $f_i:\R\too\R$ for $i=1,\ldots,n$. Then one can see easily that $h$ satisfies \eqref{eq21} and hence defines a contravariant pseudo-Hessian structure on $\R^n$.
		\item Take $M=\R^n$ endowed with its canonical affine structure and consider
		\[ h=\sum_{i,j=1}^n x_i x_j\partial_{x_i}\otimes \partial_{x_j}. \]Then one can see easily that $h$ satisfies \eqref{eq21} and hence defines a contravariant pseudo-Hessian structure on $\R^n$.
		\item Let $(M,\na)$ be an affine manifold,  $(X_1,\ldots,X_r)$ a family of parallel vector fields and $(a_{i,j})_{1\leq i,j\leq n}$ a symmetric $n$-matrix. Then
		\[ h=\sum_{i,j}a_{i,j} X_i\otimes X_j \]defines a contravariant pseudo-Hessian structure on $M$.
		
	\end{enumerate} 
\end{example}

\subsection{The Lie algebroid of a contravariant pseudo-Hessian manifold}

We show that associated to any contravariant pseudo-Hessian manifold there is a Lie algebroid structure on its cotangent bundle and a Lie algebroid flat connection.
The reader can consult \cite{fer1, mac} for more details on Lie algebroids and their connections. 
 
	 Let $(M,\na,h)$ be an affine manifold endowed with a symmetric bivector field. We  associate to this triple 
	  a bracket on $\Om^1(M)$  by putting
\begin{equation} \label{bracket}
{[\alpha,\beta]_h}:=\nabla_{h_{\sharp}(\alpha)}\beta-\nabla_{h_{\sharp}(\beta)}\alpha,\quad
\end{equation} and a map $\D:\Om^1(M)\times\Om^1(M)\too\Om^1(M)$ given by
\begin{equation} \label{eqD}
\prec \mathcal{D}_{\alpha}\beta,X\succ:=(\nabla_{X}h)(\alpha,\beta)+\prec\nabla_{h_{\sharp}(\alpha)}\beta,X\succ,\quad
\end{equation} for any  $\al,\be\in\Om^1(M)$ and $X\in\Ga(TM)$.
This bracket is skew-symmetric and satisfies obviously 
\[[\al,\be]_h=\D_\al\be-\D_\be\al \esp  [\al,f\be]_h=f[\al,\be]_h+h_\#(\al)(f)\be, \]where $ f\in C^\infty(M),\al,\be\in\Om^1(M)$.

\begin{theorem}\label{theo1} With the hypothesis and notations above, the following assertions are equivalent:
	\begin{enumerate}
		\item[$(i)$] $h$ is a  pseudo-Hessian bivector field.
		\item[$(ii)$] $(T^*M,h_\#,[\;,\;]_h)$ is a Lie algebroid.
	\end{enumerate}
	In this case, $\D$ is a connection for the Lie algebroid structure $(T^*M,h_\#,[\;,\;]_h)$ satisfying
	\[h_\#(\D_\al\be)=\na_{h_\#(\al)}h_\#(\be)\esp R_{\D}(\al,\be):=\D_{[\al,\be]_h}-\D_\al\D_\be+\D_\be\D_\al=0, \]for any $\al,\be\in\Om^1(M)$.
	\end{theorem}

\begin{proof} According to \cite[Proposition 2.1]{Be}, $(T^*M,h_\#,[\;,\;]_h)$ is a Lie algebroid if and only if, for any affine coordinates system $(x_1,\ldots,x_n)$,
	\[ h_\#([dx_i,dx_j]_h)=[h_\#(dx_i),h_\#(dx_j)]\esp \oint_{i,j,k}[dx_i,[dx_j,dx_k]_h]_h=0, \]for $1\leq i<j<k\leq n$. Since $[dx_i,dx_j]_h=0$ this is equivalent to $[h_\#(dx_i),h_\#(dx_j)]=0$ for any $1\leq i<j\leq n$ which is equivalent to \eqref{eq21}. 
	
	Suppose now that $(i)$ or $(ii)$ holds. For any, $\al,\be,\ga\in\Om^1(M)$,
	\begin{eqnarray*}
		\prec\D_\al\be, h_\#(\ga)\succ&=&\na_{h_\#(\ga)}h(\al,\be)+h(\na_{h_\#(\al)}^*\be,\ga)\\
		&=&\na_{h_\#(\al)}h(\ga,\be)+h(\na_{h_\#(\al)}^*\be,\ga)\\
		&=&h_\#(\al).h(\be,\ga)-h(\na_{h_\#(\al)}^*\ga,\be)\\
		&=&\prec\ga,\na_{h_\#(\al)}h_\#(\be)\succ.
	\end{eqnarray*}This shows that $h_\#(\D_\al\be)=\na_{h_\#(\al)}h_\#(\be)$. 
	
	Let us show now that the curvature of $\D$ vanishes. Since
	$[dx_i,dx_j]_h=0$,  it suffices to show that, for any $i,j,k\in\{1,\ldots,n\}$ with $i<j$, 
	$ \mathcal D_{dx_i}\mathcal D_{dx_j}dx_k=\mathcal D_{dx_j}\mathcal D_{dx_i}dx_k.
	$
	We have
	$$ \prec \mathcal D_{dx_i}dx_k, \frac{\partial}{\partial x_l }\succ =\frac{\partial h_{ik}}{\partial x_l}
	$$
	and hence
	$$\mathcal D_{dx_i}dx_k=\sum_{l=1}^{n} \frac{\partial h_{ik}}{\partial x_l}dx_l
	$$
	and then
	\begin{align*}
	\mathcal D_{dx_j}\mathcal D_{dx_i}dx_k &= \sum_{l=1}^{n}\mathcal D_{dx_j}\left( \frac{\partial h_{ik}}{\partial x_l}dx_l\right)\\
	&= \sum_{l=1}^{n}\left(h_\sharp(dx_j)\left(\frac{\partial h_{ik}}{\partial x_l}\right)dx_l+\frac{\partial h_{ik}}{\partial x_l}\left(\sum_{s=1}^{n}\frac{\partial h_{jl}}{\partial x_s}dx_s \right)\right)\\
	&=  \sum_{l,r}^{}h_{jr}\left(\frac{\partial^2 h_{ik}}{\partial x_r\partial x_l}\right)dx_l+\sum_{s,l}^{}\frac{\partial h_{ik}}{\partial x_l}\frac{\partial h_{jl}}{\partial x_s}dx_s\\
	&= \sum_{l,r}^{}h_{jr}\left(\frac{\partial^2 h_{ik}}{\partial x_r\partial x_l}\right)dx_l+\sum_{l,r}^{}\frac{\partial h_{ik}}{\partial x_r}\frac{\partial h_{jr}}{\partial x_l}dx_l \\
	&=\sum_{l,r}\left(h_{jr}\left(\frac{\partial^2 h_{ik}}{\partial x_r\partial x_l}\right)+\frac{\partial h_{ik}}{\partial x_r}\frac{\partial h_{jr}}{\partial x_l}\right)dx_l \\
	&=\sum_{l}\frac{\partial }{\partial x_l}\left(\sum_{r}h_{jr}\frac{\partial h_{ik}}{\partial x_r}\right)dx_l.
	\end{align*}
	So
	\begin{align*}
	\mathcal D_{dx_i}\mathcal D_{dx_j}dx_k-\mathcal D_{dx_j}\mathcal D_{dx_i}dx_k&=
	d\left(\sum_r \left(h_{jr}\frac{\partial h_{ik}}{\partial x_r} - h_{ir}\frac{\partial h_{jk}}{\partial x_r}\right)\right)
	\stackrel{\eqref{eq21}}=d(0)=0.
	\end{align*}
	\end{proof}

 The following  result is an important consequence of Theorem \ref{theo1}.
  \begin{proposition}\label{orbit} (\cite[Theorem 6.7]{Be2}) Let $(M,\na,h)$ be a contravariant pseudo-Hessian manifold. Then:
  	\begin{enumerate}
  	
  	\item The distribution $\mathrm{Im}h_\#$ is integrable and defines a singular foliation $\mathcal{L}$ on $M$. \item For any leaf $L$ of $\mathcal{L}$, $(L,\na_{|L},g_L)$ is a pseudo-Hessian manifold where $g_L$ is given by $g_L(h_\#(\al),h_\#(\be))=h(\al,\be)$. \end{enumerate}
  We will call the foliation defined by $\mathrm{Im}h_\#$ the affine foliation associated to $(M,\na,h)$.
  	
  \end{proposition}

  \begin{remark} This proposition shows that contravariant pseudo-Hessian bivector fields can be used either to build examples of affine foliations on affine manifolds or to build examples of pseudo-Hessian manifolds.
  	
  \end{remark}

 For the reader familiar with  Poisson manifolds what we have established so far shows the similarities between Poisson manifolds and contravariant pseudo-Hessian manifolds. One can consult \cite{Dufour} for more details on Poisson geometry. Poisson manifolds have many relations with Lie algebras and we will see now and in Section \ref{section4} that contravariant pseudo-Hessian manifolds are related to commutative associative algebras.
 
  Let $(M,\na,h)$ be a contravariant pseudo-Hessian manifold and $\D$  the connection given  in  \eqref{eqD}. Let $x\in M$ and  $\G_x=\ker h_{\#}(x)$.
    For any $\al,\be\in\Om^1(M)$, $h_\#(\D_\al\be)=\na_{h_\#(\al)}h_\#(\be)$. This shows that if $h_\#(\al)(x)=0$ then $h_\#(\D_\al\be)(x)=0$. Moreover, $\D_\al\be-\D_{\be}\al=\na_{h_\#(\al)}\be-\na_{h_\#(\be)}\al$. This implies that if $h_\#(\al)(x)=h_\#(\be)(x)=0$ then $\D_\al\be(x)=\D_\be\al(x)$.
  For any $a,b\in\G_x$ put
  \[ a\bullet b=(\D_\al\be)(x), \]where  $\al,\be$ are two differential 1-forms satisfying $\al(x)=a$ and $\be(x)=b$. This defines a commutative product on $\G_x$ and moreover, by using the vanishing of the curvature of $\D$, we get:
  \begin{proposition} $(\G_x,\bullet)$ is a commutative associative algebra.
  	\end{proposition}
  Near a point where $h$ vanishes, the algebra structure of $\G_x$ can be made explicit.
  \begin{proposition} \label{prass}We consider $\R^n$ endowed with its canonical affine connection,   $h$  a symmetric bivector field on $\R^n$ such that $h(0)=0$ and $(\R^n,\na,h)$ is a contravariant pseudo-Hessian manifold. Then the product on $(\R^n)^*$ given by
  	\[ e_i^*\bullet e_j^*=\sum_{k=1}^n\frac{\partial h_{ij}}{\partial x_k}(0)e_k^* \]is associative and commutative.
  	
  \end{proposition}
  \begin{proof} It is a consequence of the relation $\D_{dx_i}dx_j=dh_{ij}$ true by virtue of \eqref{eqD}.
  	\end{proof}

  \subsection{The product of contravariant pseudo-Hessian manifolds and the splitting theorem}
  
 As  the product of two Poisson manifolds is a Poisson manifold \cite{wei2}, the product of two contravariant pseudo-Hessian manifolds is a contravariant pseudo-Hessian manifold. 
  
  Let $(M_{1},\nabla^{1},h^{1})$ and $(M_{2},\nabla^{2},h^{2})$ be two contravariant pseudo-Hessian manifolds. We denote by $p_{i}:M=M_{1}\times M_{2}\rightarrow M_{i}, i=1,2$ the canonical projections. For any $X\in\Ga(TM_1)$ and $Y\in\Ga(TM_2)$, we denote by $X+ Y$ the vector field on $M$ given by $(X+ Y)(m_1,m_2)=(X(m_1),Y(m_2))$. The product of the affine atlases on $M_1$ and $M_2$ is an affine atlas on $M$ and the corresponding affine connection is the unique flat connection $\na$ on $M$ satisfying $\na_{X_1+ Y_1}(X_2+ Y_2)=\na_{X_1}^1Y_1+\na_{X_2}^2Y_2$, for any $X_1,X_2\in\Ga(TM_1)$ and $Y_1,Y_2\in\Ga(TM_2)$. Moreover, the product of $h_1$ and $h_2$ is the unique symmetric bivector field $h$ satisfying
  $$h(p_1^*\al_1,p_1^*\al_2)=h^1(\al_1,\al_2)\circ p_1,\; h(p_2^*\be_1,p_2^*\be_2)=h^2(\be_1,\be_2)\circ p_2\esp h(p_1^*\al_1,p_2^*\be_1)=0,$$
  for any $\al_1,\be_1\in\Om^1(M_1)$,  
  $\al_2,\be_2\in\Om^1(M_2)$,
  \begin{proposition}
  	 $(M,\na,h)$ is a contravariant pseudo-Hessian manifold.

  \end{proposition}
  
  \begin{proof} Let $(m_1,m_2)\in M$. Choose an affine coordinates system $(x_1,\ldots,x_{n_1})$ near $m_1$ and an affine coordinates system $(y_1,\ldots,y_{n_2})$ near $m_2$. Then
  	\[ h=\sum_{i,j}h_{ij}^1\circ p_1\partial_{x_i}\otimes \partial_{x_j}+
  	\sum_{l,k}h_{lk}^2\circ p_2\partial_{y_l}\otimes \partial_{y_k} \]and one can check easily that $h$ satisfies \eqref{eq21}.
    \end{proof}
    If we pursue the exploration of the analogies between Poisson manifolds and contravariant pseudo-Hessian manifolds we can ask naturally if there is an analog of the Darboux-Weinstein's theorem (see\cite{wei2}) in the context of contravariant pseudo-Hessian manifolds. More precisely,  let $(M,\na,h)$ be a contravariant pseudo-Hessian manifold and $m\in M$ where $\mathrm{rank}h_\#(m)=r$.
    One can ask if
     there exits an affine coordinates system $(x_1,\ldots,x_r,y_1,\ldots,y_{n-r})$ such that
    	\[ h=\sum_{i,j=1}^r h_{ij}(x_1,\ldots,x_r)\partial_{x_i}\otimes \partial_{x_j}+\sum_{i,j=1}^{n-r} f_{ij}(y_1,\ldots,y_{n-r})\partial_{y_i}\otimes \partial_{y_j}, \]
    	where $(h_{ij})_{1\leq i,j\leq r}$ is invertible and its the inverse of $\left(\frac{\partial^2\phi}{\partial x_i\partial x_j}\right)_{1\leq i,j\leq r}$ and $f_{ij}(m)=0$ for any $i,j$. Moreover, if the rank of $h_\#$ is constant near $m$ then the functions $f_{ij}$ vanish.

    The answer is no in general  for a geometric reason. Suppose that $m$ is regular, i.e., the rank of $h$ is constant near $m$ and suppose that there exists  an affine coordinates system $(x_1,\ldots,x_r,y_1,\ldots,y_{n-r})$ such that
    \[ h=\sum_{i,j=1}^r h_{ij}(x_1,\ldots,x_r)\partial_{x_i}\otimes \partial_{x_j}. \]
    This will have a strong geometric consequence, namely that $\mathrm{Im}h_\#=\mathrm{span}(\partial_{x_1},\ldots,\partial_{x_r})$ and the associated affine foliation is parallel, i.e., if $X$ is a local vector field and $Y$ is tangent to the foliation then $\na_XY$ is tangent to the foliation. We give now an example of a regular contravariant pseudo-Hessian manifold whose associated affine foliation is not parallel which shows that the analog of Darboux-Weinstein is not true in general.
\begin{example}
		 We consider $M=\R^4$ endowed with its canonical affine connection $\na$, denote by $(x,y,z,t)$ its canonical coordinates and consider 
		\[ X=\cos(t)\partial_x+\sin(t)\partial_y+\partial_z,\;Y=-\sin(t)\partial_x+\cos(t)\partial_y\esp h=X\otimes Y+Y\otimes X. \]
		We have $\na_XX=\na_YX=\na_XY=\na_YY=0$ and hence $h$ is a  pseudo-Hessian bivector field,  $\mathrm{Im}h_\#=\mathrm{span}\{X,Y \}$ and the rank of $h$ is constant equal to 2. However, the foliation associated to $\mathrm{Im}h_\#$ is not parallel since $\nabla_{\partial_t}Y=-X+\partial_z\notin \mathrm{Im}h_\#$.  
		
\end{example}

However, when $h$ has constant rank equal to $\dim M-1$, we have the following result and its important corollary.

     \begin{theorem}\label{th2}
     	Let $(M,\nabla,h)$ be a contravariant pseudo-Hessian manifold and $m{\in}M$ such that $m$ is a regular point and the rank of $h_{\#}(m)$ is equal to $n-1$. Then there exists an affine coordinates system $(x_{1},\ldots,x_{n})$ around $m$ and a function $f(x_{1},\ldots,x_{n})$ such that
     	\begin{equation*}
     	h=\sum_{i,j=1}^{n-1}h_{ij}\partial_{x_{i}}\otimes\partial_{x_{j}},
     	\end{equation*} 
     	and the matrix $(h_{ij})_{1\leq i,j\leq n-1}$ is invertible and its inverse is the matrix $\big(\frac{\partial^{2}f}{\partial x_{i}\partial x_{j}}\big)_{1\leq i,j\leq n-1}$.
     \end{theorem}
     
     \begin{corollary}\label{copar} Let $(M,\na,h)$ be a contravariant pseudo-Hessian manifold with $h$ of constant rank equal to $\dim M-1$. Then the affine foliation associated to $\mathrm{Im}h_\#$ is $\na$-parallel.
     	
     \end{corollary}
     In order to prove this theorem, we need the following lemma.
     \begin{lemma}\label{le} Let $f:\R^2\too\R$ be a differentiable function such that $\partial_x(f)+f\partial_{y}(f)=0$. Then $f$ is a constant.
     	
     \end{lemma}
     \begin{proof} Let $f$ be a solution of the equation above. We consider the vector field $X_f=\partial_x+f\partial_y$. The integral curve $(x(t),y(t))$ of $X_f$ passing through $(a,b)\in\R^2$ satisfies
     	\[ x'(t)=1,\quad y'(t)=f(x(t),y(t))\esp (x(0),y(0))=(a,b). \]Now
     	\[ y''(t)=\partial_x(f)(x(t),y(t))+y'(t)\partial_y(f)(x(t),y(t))=0 \] and hence, the flow of $X_f$ is given by $\phi(t,(x,y))=(t+x,f(x,y)t+y)$. The relation $\phi(t+s,(x,y))=\phi(t,\phi(s,(x,y)))$ implies that the map
     	 $F(x,y)=(1,f(x,y))$ satisfies
     	\[ F(u+tF(u))=F(u),\;u\in\R^2,t\in\R. \] Let $u,v\in\R^2$ such that $F(u)$ and $F(v)$ are linearly independent. Then there exists $s,t\in\R$ such that $u-v=tF(u)+sF(v)$ and hence $F(u)=F(v)$ which is a contradiction. So $F(x,y)=\al(x,y)(a,b)$, i.e., $(1,f(x,y))=(\al(x,y)a,\al(x,y)b)$ and $\al$ must be constant and hence $f$ is constant.
     \end{proof}
     
    \paragraph{Proof of Theorem \ref{th2}}\begin{proof} Let $(x_1,\ldots,x_n)$ be an affine coordinates system near $m$ such that $(X_1,\ldots,X_{n-1})$ are linearly independent in a neighborhood of $m$, where $X_i=h_\#(dx_i)$, $X_n=\di\sum_{j=1}^{n-1}f_{j}X_j$ and, by virtue of the proof of Theorem \ref{theo1}, for any $1\leq i<j\leq n$, $[X_i,X_j]=0$. 
    For any $i=1,\ldots,n-1$, the relation $[X_i,X_n]=0$ is equivalent to 
    \[ X_i(f_j)=h_{in}\partial_{x_n}(f_j)+\sum_{l=1}^{n-1}h_{il}\partial_{x_l}(f_j)=0,\;\quad j=1,\ldots,n-1. \]
    But $h_{in}=X_n(x_i)=\di\sum_{l=1}^{n-1}f_lh_{il}$ and hence, for any $i,j=1,\ldots n-1$,
    \[ \sum_{l=1}^{n-1}h_{il}(f_l\partial_{x_n}(f_j)+\partial_{x_l}(f_j))=0. \]
   Or the matrix $(h_{ij})_{1\leq i,j\leq n-1}$ is invertible so we get
      \begin{equation}\label{c} f_l\partial_{x_n}(f_j)+\partial_{x_l}(f_j)=0,\quad l,j=1,\ldots,n-1. \end{equation}For $l=j$ we get that $f_j$ satisfies
    $f_j\partial_{x_n}(f_j)+\partial_{x_j}(f_j)=0$ so,  according to  Lemma \ref{le}, $\partial_{x_n}(f_j)=\partial_{x_j}(f_j)=0$ and, from \eqref{c}, $f_j=$constant. We consider $y=f_1x_1+\ldots+f_{n-1}x_{n-1}-x_n$, we have $h_\#(dy)=0$ and $(x_1,\ldots,x_{n-1},y)$ is an affine coordinates system around $m$.

    On the other hand,  there exists a coordinates system $(z_{1},\ldots,z_{n})$ such that
    \begin{equation*}
    h_{\sharp}(dx_{i})=\partial_{z_{i}},\text{ }i=1,\ldots,n-1.	
    \end{equation*} 
    We deduce that
    \begin{equation*}
    \partial_{x_{i}}=\sum_{j=1}^{n-1}h^{ij}\partial_{z_{j}},\text{ }i=1,\ldots,n-1,	
    \end{equation*}  
    with $h^{ij}=\frac{\partial z_{j}}{\partial x_{i}}$. We consider $\sigma=\sum_{j=1}^{n-1}z_{j}dx_{j}$. We have $d\sigma=0$ so according to the foliated Poincar\'e Lemma (see\cite[p.56]{Ca}) there exists a function $f$ such that $h^{ij}=\frac{\partial^{2}f}{\partial x_{i}\partial x_{j}}$.
    \end{proof}

    \subsection{The divergence and the modular class of a contravariant pseudo-Hessian manifold}
    
    We define now the divergence of a contravariant pseudo-Hessian structure.
    We recall first the definition of the divergence of multivector fields associated to a connection on a manifold.
    
    Let $(M,\nabla)$ be a manifold endowed with a connection. We define $\mathrm{div}_{\nabla}:\Gamma(\otimes^{p}TM)\rightarrow\Gamma(\otimes^{p-1}TM)$ by 
    \begin{equation*}
    \mathrm{div}_{\nabla}(T)(\alpha_{1},\ldots,\alpha_{p-1})=\sum_{i=1}^{n}\nabla_{e_{i}}(T)(e_{i}^{*},\alpha_{1},\ldots,\alpha_{p-1}),	
    \end{equation*}
    where $\alpha_{1},\ldots,\alpha_{p-1}\in T^{*}_{x}M$, $(e_{1},\ldots,e_{n})$ a basis of $T_{x}M$ and $(e_{1}^{*},\ldots,e_{n}^{*})$ its dual basis. This operator respects the symmetries of  tensor fields.

    Suppose now that $(M,\nabla,h)$ is a contravariant pseudo-Hessian manifold. The divergence of this structure is the vector field $\mathrm{div}_{\nabla}(h)$. 
    This vector field is an invariant of the pseudo-Hessian structure and has an important property. Indeed, let $\mathbf{d}_h:\Ga(\wedge^\bullet TM)\too \Ga(\wedge^{\bullet+1} TM)$ be the differential associated to the Lie algebroid structure $(T^*M,h_\#,\br_h)$ and given by
    \begin{eqnarray*}
    \mathbf{d}_hQ(\al_1,\ldots,\al_{p})&=&\sum_{j=1}^{p}(-1)^{j+1}h_\#(\al_j).Q(\al_1,\ldots,
    \hat{\al_j},\ldots\al_{p})\\&&
    +\sum_{1\leq i<j\leq p}(-1)^{i+j}Q([\al_i,\al_j]_h,\al_1,\ldots,\hat{\al_i},\ldots,\hat{\al_j},\ldots,\al_p).
    \end{eqnarray*}
    
    \begin{proposition}\label{co}$\mathbf{d}_h(\mathrm{div}_{\na}(h))=0$.\end{proposition} 
    \begin{proof} Let $(x_1,\ldots,x_n)$ be an affine coordinates system. We have
    \begin{eqnarray*}
    	\mathbf{d}_h\mathrm{div}_\na(h)(\al,\be)&=&\sum_{i=1}^n\left( h_\#(\al).\na_{\partial_{x_i}}(h)(dx_i,\be)-
    	h_\#(\be).\na_{\partial_{x_i}}(h)(dx_i,\al)
    	-\na_{\partial_{x_i}}(h)(dx_i,\na_{h_\#(\al)}\be)+\na_{\partial_{x_i}}(h)(dx_i,\na_{h_\#(\be)}\al)\right)\\
    	&=&\sum_{i=1}^n\left(\na_{h_\#(\al)}\na_{\partial_{x_i}}(h)(dx_i,\be)-\na_{h_\#(\be)}\na_{\partial_{x_i}}(h)(dx_i,\al)\right)\\
    	&\stackrel{\eqref{codazzib}}=&\sum_{i=1}^n\left(\na_{[h_\#(\al),\partial_{x_i}]}(h)(dx_i,\be)-\na_{[h_\#(\be),\partial_{x_i}]}(h)(dx_i,\al)\right).
    \end{eqnarray*}If we take $\al=dx_l$ and $\be=dx_k$, we have
    \[ [\partial_{x_i},h_\#(dx_l)]=\sum_{m=1}^n\partial_{x_i}(h_{ml})\partial_{x_m} \]
    and hence
    \begin{eqnarray*}
    	\mathbf{d}_h\mathrm{div}_\na(h)(\al,\be)&=&\sum_{i,m=1}^n\left(\partial_{x_i}(h_{ml})\partial_{x_m}(h_{ik}) -\partial_{x_i}(h_{mk})\partial_{x_m}(h_{il})    \right)=0.
    \end{eqnarray*}\end{proof}
    
    Let $(M,\na,h)$ be an orientable contravariant pseudo-Hessian manifold and $\Om$ a volume form on $M$. For any $f$ we denote by $X_f=h_\#(df)$ and we define $\mathbf{M}_\Om:C^\infty(M,\R)\too C^\infty(M,\R)$ by putting for any $f\in C^\infty(M,\R)$,
    \[ \na_{X_f}\Om=\mathbf{M}_\Om(f)\Om. \]
    It is obvious that $\mathbf{M}_\Om$ is a derivation and hence a vector field and
    $\mathbf{M}_{e^f\Om}=X_f+\mathbf{M}_\Om$.
    Moreover, if $(x_1,\ldots,x_n)$ is an affine coordinates system and $\mu=\Om(\partial_{x_1},\ldots,\partial_{x_n})$ then
    \[ \na_{X_f}\Om(\partial_{x_1},\ldots,\partial_{x_n})=X_f(\mu)=X_{\ln|\mu|}(f)\mu. \] So in the coordinates system $(x_1,\ldots,x_n)$, we have $\mathbf{M}_\Om=X_{\ln|\mu|}$. This implies $\mathbf{d}_h\mathbf{M}_\Om=0.$
    The cohomology class of $\mathbf{M}_\Om$ doesn't depend on $\Om$ and we call it the {\it modular class} of $(M,\na,h)$.
    \begin{proposition} The modular class of $(M,\na,h)$ vanishes if and only if there exists a volume form $\Om$ such that $\na_{X_f}\Om=0$ for any $f\in C^\infty(M,\R)$.
    	\end{proposition}
   By analogy with the case of Poisson manifolds, one can ask if it is possible to find a volume form $\Om$ such that $\Li_{X_f}\Om=0$ for any $f\in C^\infty(M,\R)$. The following proposition gives a negative answer to this question unless $h=0$.
    
    \begin{proposition} Let $(M,\na,h)$ be an orientable contravariant pseudo-Hessian manifold. Then:
    	\begin{enumerate}
    		\item For any volume form $\Om$ and any $f\in C^\infty(M,\R)$,
    		\[ \Li_{X_f}\Om=\left[ \mathbf{M}_\Om(f)+\dev_\na(h)(f)+\prec h,\mathrm{Hess}(f)\succ   \right]\Om, \]where $\mathrm{Hess}(f)(X,Y)=\na_X(df)(Y)$ and $\prec h,\mathrm{Hess}(f)\succ$
    		is the pairing between the bivector field $h$ and the 2-form $\mathrm{Hess}(f)$.
    		\item There exists a volume form $\Om$ such that $\Li_{X_f}\Om=0$ for any $f\in C^\infty(M,\R)$ if and only if $h=0$.
    	\end{enumerate}
    	
    \end{proposition}
    
    \begin{proof}\begin{enumerate}
    		\item 
    	 Let $(x_1,\ldots,x_n)$ be an affine coordinates system. Then:
    \begin{eqnarray*}
    	\;[X_f,\partial_{x_i}]&=&\sum_{l,j=1}^n[\partial_{x_j}(f)h_{jl}\partial_{x_l},\partial_{x_i}]\\
    	&=&-\sum_{l,j=1}^n\left(h_{jl}\partial_{x_i}\partial_{x_j}(f)+\partial_{x_j}(f)\partial_{x_i}(h_{jl})   \right)\partial_{x_l},\\
    \Li_{X_f}\Om(\partial_{x_1},\ldots,\partial_{x_n})&=&(\na_{X_f}\Om)(\partial_{x_1},\ldots,\partial_{x_n})-\sum_{i=1}^n\Om((\partial_{x_1},\ldots,[X_f,\partial_{x_i}],\ldots,\partial_{x_n}))\\
    &=&(\na_{X_f}\Om)(\partial_{x_1},\ldots,\partial_{x_n})+
    \sum_{i,j=1}^n\left(h_{ji}\partial_{x_i}\partial_{x_j}(f)+\partial_{x_j}(f)\partial_{x_i}(h_{ji})   \right)\Om(\partial_{x_1},\ldots,\partial_{x_n})
    \end{eqnarray*}	and the formula follows since $\dev_\na(h)=\di\sum_{i,j=1}^n\partial_{x_i}(h_{ji})\partial_{x_j}$.
    \item This is a consequence of the fact that $\mathbf{M}_\Om$ and $\dev_\na(h)$ are derivation and
    \[ \prec h,\mathrm{Hess}(fg)\succ=f\prec h,\mathrm{Hess}(g)\succ+g\prec h,\mathrm{Hess}(f)\succ+\prec h,df\odot dg\succ. \]
    
    	\end{enumerate}	
    \end{proof}

    \section{The tangent bundle of a contravariant pseudo-Hessian manifold}\label{section3}
    
    In this section, we define and study the associated Poisson tensor on the tangent bundle of a contravariant pseudo-Hessian manifold. One can see \cite{dom} for the classical properties of the tangent bundle of a manifold endowed with a linear connection.
    
    Let $(M,\na)$ be an affine manifold, $p:TM\too M$ the canonical projection and $K:TTM\too TM$ the connection map of $\na$ locally given   by
    \[ K\left( \sum_{i=1}^n b_i\partial_{x_i}+\sum_{j=1}^nZ_j\partial_{\mu_j}\right)=
    \sum_{l=1}^n\left( Z_l+\sum_{i=1}^n\sum_{j=1}^n b_i\mu_j\Ga_{ij}^l\right)\partial_{x_l}, \]where $(x_1,\ldots,x_n)$ is a system of local coordinates,   $(x_1,\ldots,x_n,\mu_1,\ldots,\mu_n)$ the associated system of coordinates on $TM$ and $\na_{\partial_{x_i}} \partial_{x_j}=\sum_{l=1}^n\Ga_{ij}^l \partial_{x_l}$. Then
    \[ TTM=\ker Tp\oplus \ker K. \]
     For  $X\in\Ga(TM)$, we denote by $X^h$ its horizontal lift and by $X^v$ its vertical lift.  The flow of $X^v$ is given by $\Phi^X(t,(x,u))=(x,u+tX(x))$ and  $X^h(x,u)=h^{(x,u)}(X(x))$, where $h^{(x,u)}:T_xM\too \ker K(x,u)$ is the inverse of the restriction of $dp$ to $\ker K(x,u)$. Since the curvature of $\na$ vanishes, we have
    \begin{equation}\label{br} [X^h,Y^h]=[X,Y]^h,\;[X^h,Y^v]=(\na_XY)^v \esp[X^v,Y^v]=0, \end{equation}for any $X,Y\in\Ga(TM)$. For any $\al\in\Om^1(M)$, we define $\al^v,\al^h\in\Om^1(TM)$  by
    \[ \begin{cases}
    \al^v(X^v)=\al(X)\circ p,\\\al^v(X^h)=0,
    \end{cases}\esp \begin{cases}
    \al^h(X^h)=\al(X)\circ p,\\\al^h(X^v)=0.
    \end{cases} \]
    The following proposition is well-known \cite{dom} and can be proved easily.
  \begin{proposition} The connection $\overline{\na}$  on  $TM$ given  by
  	\begin{equation}\label{con} \overline{\na}_{X^h}Y^h=(\na_XY)^h,\;\overline{\na}_{X^h}Y^v=(\na_XY)^v\esp\overline{\na}_{X^v}Y^h=\overline{\na}_{X^v}Y^v=0, \end{equation}where $X,Y\in\Ga(TM)$, defines an affine structure on $TM$. Moreover, the endomorphism vector field $J:TTM\too TTM$ given by $JX^h=X^v$ and $JX^v=-X^h$ satisfies $J^2=-\mathrm{Id}_{TTM}$,  is parallel with respect to $\overline{\na}$ and hence defines a complex structure on $TM$.
  	
  \end{proposition}

    Let $h$ be a symmetric bivector field on $M$. We associate to $h$ a skew-symmetric bivector field $\Pi$  on $TM$  by putting
    \[ 
    \Pi(\al^v,\be^v)=\Pi(\al^h,\be^h)=0\esp 
    \Pi(\al^h,\be^v)=-\Pi(\be^v,\al^h)=h(\al,\be)\circ p,
    \]for any $\al,\be\in\Om^1(M)$. 
    For any  $\al\in\Om^1(M)$, 
    \begin{equation}\label{eqpi}
   \Pi_\#(\al^v)=-h_\#(\al)^h\esp \Pi_\#(\al^h)=h_\#(\al)^v. \end{equation}
    To prove one of our main result in this section, we need the following proposition which is a part of the folklore.
    \begin{proposition}\label{folk} Let $(P,\na)$ be a manifold endowed with a torsionless connection and $\pi$ is a bivector field on $P$. Then the Nijenhuis-Schouten bracket $[\pi,\pi]$ is given by 
    	\begin{eqnarray*}
    		\;[\pi,\pi](\al,\be,\ga)
    		&=&2\left( \na_{\pi_\#(\al)}\pi(\be,\ga)+\na_{\pi_\#(\be)}\pi(\ga,\al)+\na_{\pi_\#(\ga)}\pi(\al,\be) \right).
    	\end{eqnarray*}
    	
    \end{proposition}
    
    \begin{theorem}\label{poisson} The following assertions are equivalent:
    	\begin{enumerate}
    		\item[$(i)$] $(M,\na,h)$ is a contravariant pseudo-Hessian manifold.
    		\item[$(ii)$] $(TM,\Pi)$ is a Poisson manifold.\end{enumerate}
    In this case, if 	 $L$ is a   leaf of $\mathrm{Im}h_\#$ then $TL\subset TM$ is a symplectic leaf of $\Pi$  which is also a complex submanifold of $TM$. Moreover, if $\om_L$ is the symplectic form of $TL$ induced by $\Pi$ and $g_L$ is the pseudo-Riemannian metric given by $g_L(U,V)=\om(JU,V)$ then $(TL,g_L,\om_L,J)$ is a pseudo-K\"ahler manifold.

    \end{theorem}

    \begin{proof} We will use Proposition \ref{folk} to prove the equivalence. Indeed, by a direct computation one can establish easily, for any $\al,\be,\ga\in\Om^1(M)$, the following relations
    	\begin{eqnarray*}
    	\overline{\na}_{\Pi_\#(\al^v)}\Pi(\be^v,\ga^v)&=&\overline{\na}_{\Pi_\#(\al^v)}\Pi(\be^h,\ga^h)=\overline{\na}_{\Pi_\#(\al^h)}\Pi(\be^v,\ga^v)=\overline{\na}_{\Pi_\#(\al^h)}\Pi(\be^h,\ga^h)=
    	\overline{\na}_{\Pi_\#(\al^h)}\Pi(\be^h,\ga^v)=0,\\
    	\overline{\na}_{\Pi_\#(\al^v)}\Pi(\be^h,\ga^v)
    	&=&\na_{h_\#(\al)}(h)(\be,\ga)\circ p,
    	\end{eqnarray*}and the equivalence follows. The second part of the theorem is obvious and the only point which need to be checked is that $g_L$ is  nondegenerate.
    	\end{proof}
    
  \begin{remark} \label{rem1}\begin{enumerate}\item The total space of the dual of a Lie algebroid carries a Poisson tensor (see \cite{mac}).
  		 If $(M,\na,h)$ is a contravariant pseudo-Hessian manifold then, according to Theorem \ref{theo1}, $T^*M$ carries a Lie algebroid structure and one can see easily that $\Pi$ is the corresponding Poisson tensor on $TM$.
  		\item The equivalence of $(i)$ and $(ii)$ in Theorem \ref{poisson} deserves to be stated explicitly in the case of $\R^n$ endowed with its canonical affine structure $\na$. Indeed, let $(h_{ij})_{1\leq i,j\leq n}$ be a symmetric matrix where $h_{ij}\in C^\infty(\R^n,\R)$ and $h$ the associated symmetric bivector field on $\R^n$. The associated bivector field $\Pi_h$ on $T\R^n=\C^n$ is    		\[ \Pi_h=\sum_{i,j=1}^nh_{ij}(x)\partial_{x_i}\wedge\partial_{y_j}, \] where $(x_1+iy_1,\ldots,x_n+iy_n)$ are the canonical coordinates of $\C^n$. Then, according to Theorem \ref{theo1}, $(\R^n,\na,h)$ is a contravariant pseudo-Hessian manifold if and only if $(\C^n,\Pi_h)$ is a Poisson manifold.
  	\end{enumerate}

  \end{remark}
  
  We explore now some relations between some invariants of $(M,\na,h)$ and some invariants of $(TM,\Pi)$. 
  
  \begin{proposition} Let $(M,\na,h)$ be a contravariant pseudo-Hessian manifold. Then $(\mathrm{div}_{\na}h)^v=\mathrm{div}_{\overline{\na}}\Pi.$
  	
  \end{proposition}

  \begin{proof} Fix $(x,u)\in TM$ and choose a basis $(e_1,\ldots,e_n)$ of $T_xM$. Then $(e_1^v,\ldots,e_n^v,e_1^h,\ldots,e_n^h)$ is a basis of $T_{(x,u)}TM$ with
  	$((e_1^*)^v,\ldots,(e_n^*)^v,(e_1^*)^h,\ldots,(e_n^*)^h)$ as a dual basis. For any $\al\in T_x^*M$, we have
  	\begin{eqnarray*}
  	\prec\al^v,\mathrm{div}_{\overline{\na}}\Pi\succ&=&\sum_{i=1}^n\left( \overline{\na}_{e_i^v}(\Pi)((e_i^*)^v,\al^v)+
  	\overline{\na}_{e_i^h}(\Pi)((e_i^*)^h,\al^v) \right)\\
  	&\stackrel{\eqref{con}}=&\prec\al,\mathrm{div}_\na(h)\succ\circ p=\prec\al^v,(\mathrm{div}_\na(h))^v\succ.
  	\end{eqnarray*}In the same way we get that $\prec\al^h,\mathrm{div}_{\overline{\na}}\Pi\succ=0$ and the result follows.
  	\end{proof}

  Let $(M,\na,h)$ be a contravariant pseudo-Hessian manifold. For any multivector field $Q$ on $M$ we define its vertical lift $Q^v$ on $TM$ by
  \[ i_{\al^h}Q^v=0\esp Q^v(\al_1^v,\ldots,\al_q^v)=Q(\al_1,\ldots,\al_q)\circ p. \]
  Recall that $h$ defines a Lie algebroid structure on $T^*M$ whose anchor is $h_\#$ and the Lie bracket is given by \eqref{bracket}. The Poisson tensor $\Pi$ defines a Lie algebroid structure on $T^*TM$ whose anchor is $\Pi_\#$ and the Lie bracket is the Koszul bracket
  \[ [\phi_1,\phi_2]_\Pi=\Li_{\Pi_\#(\phi_1)}\phi_2-\Li_{\Pi_\#(\phi_2)}\phi_1-d\Pi(\phi_1,\phi_2),\quad\phi_1,\phi_2\in\Om^1(TM). \]
  We denote by $\mathbf{d}_h$ (resp. $\mathbf{d}_\Pi$) the differential  associated to the Lie algebroid structure on $T^*M$ (resp. $T^*TM$) defined by $h$ (resp. $\Pi$).  
  
  \begin{proposition}\label{coho}\begin{enumerate}
  		\item[$(i)$] For any $\al,\be\in\Om^1(M)$ and $X\in\Ga(TM)$,
  		\[ \begin{cases}
  		\Li_{X^h}\al^h=(\Li_X\al)^h,\;\Li_{X^h}\al^v=(\na_X\al)^v,\;\Li_{X^v}\al^h=0\esp 
  		\Li_{X^v}\al^v=(\Li_X\al)^h-(\na_X\al)^h,\\
  		[\al^h,\be^h]_\Pi=0,\;[\al^v,\be^v]_\Pi=-[\al,\be]_h^v\esp [\al^h,\be^v]_\Pi=(\D_\be\al)^h,
  		\end{cases} \]where $\D$ is the connection given by \eqref{eqD}.
  		\item[$(ii)$] $(\mathbf{d}_hQ)^v=-\mathbf{d}_\Pi(Q^v)$.
  	\end{enumerate}
  	 \end{proposition}
  	 \begin{proof} The relations in $(i)$ can be established by a straightforward computation. From these relations and the fact that $\Pi_\#(\al^h)=(h_\#(\al))^v$ one can deduce easily that $i_{\al^h}\mathbf{d}_\Pi(Q^v)=0$. On the other hand, since $\Pi_\#(\al^v)=-(h_\#(\al))^h$ and $[\al^v,\be^v]_\Pi=-[\al,\be]^v$ we can conclude.
  	 	\end{proof}
  	 	
  	 	\begin{remark} From Propositions \ref{co} and Proposition \ref{coho}, we can deduce that $\mathbf{d}_\Pi(\mathrm{div}_{\overline{\na}}\Pi)=0$. This is not a surprising result because $\overline{\na}$ is flat and $\mathrm{div}_{\overline{\na}}\Pi$ is a representative of the modular class of $\Pi$.
  	 		
  	 	\end{remark}

  As a consequence of Proposition \ref{coho} we can define a linear map from the cohomology of $(T^*M,h_\#,[\;,\;]_h)$ to the cohomology of $(T^*TM,\Pi_\#,[\;,\;]_\Pi)$ by
  \[ V:H^*(M,h)\too H^*(TM,\Pi),\; [Q]\mapsto[Q^v]. \]
  
  \begin{proposition} $V$ is injective.
  	
  \end{proposition}
  \begin{proof} An element $P\in \Ga(\wedge^d TTM)$ is of type $(r,d-r)$ if for any $q\not=r$  
  	\[ P(\al_1^v,\ldots,\al_{q}^v,\be_1^h,\ldots,\be_{d-q}^h)=0, \]for any $\al_1,\ldots,\al_q,\be_1,\ldots,\be_{d-q}\in\Om^1(M)$. We have
  	\[ \begin{cases}\Ga(\wedge^d TTM)=\bigoplus_{r=0}^d\Ga_{(r,d-r)}(\wedge^d TTM),\\  
  	\mathbf{d}_\Pi(\Ga_{(r,d-r)}(\wedge^{d} TTM))\subset \Ga_{(r+1,d-r)}(\wedge^{d+1} TTM)\oplus
  	\Ga_{(r,d+1-r)}(\wedge^{d+1} TTM).\end{cases} \]

  	Let $Q\in\Ga(\wedge^d TM)$ such that $\mathbf{d}_hQ=0$ and there exists $P\in \Ga(\wedge^{d-1} TTM)$ such that $\mathbf{d}_\Pi P=Q^v$. Since $Q^v\in \Ga_{(d,0)}(\wedge^d TTM) $ then $P\in \Ga_{(d-1,0)}(\wedge^{d-1} TTM)$. Let us show that $P=T^v$. For $\al_1,\ldots,\al_{d-1},\be\in\Om^1(M)$, we have
  	\begin{eqnarray*}
  	0&=&\mathbf{d}_\Pi P(\be^h,\al_1^v,\ldots,\al_{d-1}^v)
  	=(h_\#(\be))^v.P(\al_1^v,\ldots,\al_{d-1}^v).
  	\end{eqnarray*}So the function $P(\al_1^v,\ldots,\al_{d-1}^v)$ is constant on the fibers of $TM$ and hence there exists $T\in\Ga(\wedge^{d-1}TM)$ such that $P(\al_1^v,\ldots,\al_{d-1}^v)=T(\al_1,\ldots,\al_{d-1})\circ p$. So $[Q]=0$ which completes the proof.
  	\end{proof}

  \section{Linear,  affine and multiplicative  contravariant pseudo-Hessian structures}\label{section4}
  
  \subsection{Linear and  affine   contravariant pseudo-Hessian structures}
  
  As in the Poisson geometry context, we have the notions of linear and  affine   contravariant pseudo-Hessian structures. One can see \cite{kassel} for the notion of cocycle in associative algebras.
  
   Let $(V,\na)$ be a finite dimensional real vector space endowed with its canonical affine structure. A symmetric bivector field $h$ on $V$ is called affine if  there exists a commutative product $\bullet$ on $V^*$ and a symmetric bilinear form $B$ on $V^*$ such that, for any $\al,\be\in V^*\subset\Om^1(V)$ and $u\in V$,
  	\[ h(\al,\be)(u)=\prec \al\bullet\be,u\succ+B(\al,\be). \]
  	One can see easily that if $\al,\be\in\Om^1(V)=C^\infty(V,V^*)$ then
  	\[ h(\al,\be)(u)=\prec \al(u)\bullet\be(u),u\succ+B(\al(u),\be(u)). \]
  	If $B=0$, $h$ is called linear.
  	
  	If $(x_1,\ldots,x_n)$ is a linear coordinates system on $V^*$ associated to a basis $(e_1,\ldots,e_n)$ then
  	\[ h(dx_i,dx_j)=b_{ij}+\sum_{k=1}^nC_{ij}^kx_k, \]where $e_i\bullet e_j=\sum_{k=1}^nC_{ij}^ke_k$ and $b_{ij}=B(e_i,e_j)$.
  
  \begin{proposition}
  	  $(V,\na,h)$ is a contravariant pseudo-Hessian manifold if and only if
  	$\bullet$ is associative and $B$ is a scalar 2-cocycle of $(V^*,\bullet)$, i.e.,
  	\[ B(\al\bullet \be,\ga)=B(\al,\be\bullet \ga) \]for any $\al,\be,\ga\in V^*$.
  	
  \end{proposition}
  
  \begin{proof} For any $\al\in V^*$ and $u\in V$, $h_\#(\al)(u)=\mathrm{L}_\al^*u+i_\al B$ where $\mathrm{L}_\al(\be)=\al\bullet\be$ and $i_\al B\in V^{**}=V$. We denote by $\phi^{h_\#(\al)}$ the flow of the vector field $h_\#(\al)$. 
  	Then, for any $\al,\be,\ga\in V^*$,
  	\begin{eqnarray*} \na_{h_\#(\al)}(h)(\be,\ga)(u)&=&\frac{d}{dt}_{|t=0}\left( \prec\be\bullet\ga,\phi^{h_\#(\al)}(t,u)\succ+B(\be,\ga)   \right)\\
  		&=&\prec \be\bullet\ga,\mathrm{L}_\al^*u+i_\al B\succ\\
  		&=&\prec\al\bullet(\be\bullet\ga),u\succ+B(\al,\be\bullet\ga) \end{eqnarray*}and the result follows.
  	\end{proof}
  
  Conversely, we have the following result.
  
  \begin{proposition} Let $(\mathcal{A},\bullet,B)$ be a commutative and associative algebra endowed with a symmetric scalar 2-cocycle. Then:
  	\begin{enumerate}
  		\item  $\mathcal{A}^*$ carries a structure of a contravariant pseudo-Hessian structure $(\na,h)$  where $\na$ is the canonical affine structure of $\mathcal{A}^*$ and $h$ is given by
  	\[ h(u,v)(\al)=\prec\al, u(\al)\bullet v(\al)\succ+B(u(\al),v(\al)),\quad\al\in \mathcal{A}^*,u,v\in\Om^1(\mathcal{A}^*). \]\item When $B=0$, the leaves of the affine foliation  associated to $\mathrm{Im}h_\#$ are the orbits of the action $\Phi$ of $(\mathcal{A},+)$ on $\mathcal{A}^*$ given by
  	$\Phi(u,\al)=\exp(\mathrm{L}_u^*)(\al)$ \item  The associated Poisson tensor $\Pi$ on $T\mathcal{A}^*=\mathcal{A}^*\times \mathcal{A}^*$ is the affine Poisson tensor dual associated to the Lie algebra $(\mathcal{A}\times \mathcal{A},\br)$ endowed with the 2-cocycle $B_0$ where
  	\[ [(a,b),(c,d)]=(a\bullet d-b\bullet c,0)\esp B_0((a,b),(c,d))=B(a,d)-B(c,b). \]
  \end{enumerate}
  \end{proposition}
  
  \begin{proof} It is only the third point which need to be checked. One can see easily that $\br$ is a Lie bracket on $\mathcal{A}\times \mathcal{A}$ and $B_0$ is a scalar 2-cocycle for this Lie bracket. For any $a\in \mathcal{A}\subset \Om^1(\mathcal{A}^*)$, $a^v=(0,a)\in \mathcal{A}\times \mathcal{A}\subset \Om^1(\mathcal{A}^*\times \mathcal{A}^*)$ and $a^h=(a,0)$. So
  	\[ \Pi(a^h,b^v)(\al,\be)=h(a,b)(\al)=\prec \al,a\bullet b\succ+B(a,b). \]
  	On the the other hand, if $\Pi^*$ is the Poisson tensor dual, then
  	\begin{eqnarray*} \Pi^*(a^h,b^v)(\al,\be)&=&\Pi^*((a,0),(0,b))(\al,\be)\\
  		&=&\prec (\al,\be),[(a,0),(0,b)]\succ+B_0((a,0),(0,b))\\
  		&=&\prec\al,a\bullet b\succ+B(a,b)\\
  		&=&\Pi(a^h,b^v)(\al,\be).
  		\end{eqnarray*}In the same way one can check the others equalities.
  \end{proof}
  
  This proposition can be used as a machinery to build examples of pseudo-Hessian manifolds. Indeed,  by virtue of Proposition \ref{orbit}, any orbit $L$ of the action $\Phi$ has an affine structure $\na_L$ and a pseudo-Riemannian metric $g_L$ such that $(L,\na_L,g_L)$ is a pseudo-Hessian manifold.
  
  \begin{example} We take $\mathcal{A}=\R^4$  with its canonical basis $(e_i)_{i=1}^4$ and $(e_i^*)_{i=1}^4$ is the dual basis. We endow $\mathcal{A}$  with the commutative associative product given by
  	\[ e_1\bullet e_1=e_2,\; e_1\bullet e_2=e_3,\; e_1\bullet e_3= e_2\bullet e_2=e_4, \]the others products are zero and we endow $\mathcal{A}^*$ with the linear contravariant pseudo-Hessian structure associated to $\bullet$.
  	 	We denote by $(a,b,c,d)$ the linear coordinates on $\mathcal{A}$ and $(x,y,z,t)$ the dual coordinates on $\mathcal{A}^*$.  
  	 We have
  	\[ \Phi(ae_1+be_2+ce_3+de_4^*,xe_1^*+ye_2^*+ze_3^*+te_4^*)=(x+ay+(\frac12a^2+b)z+(\frac16a^3+ab+c)t,y+az+(\frac12a^2+b)t,z+at,t) \]and
  	\[ X_{e_1}=y\partial_x+z\partial_y+t\partial_z,\; X_{e_2}=z\partial_x+t\partial_y,\; X_{e_3}=t\partial_x\esp X_{e_4}=0. \]
  	Let us describe the pseudo-Hessian structure of the hyperplane $M_c=\{t=c,c\not=0\}$ endowed with the coordinates $(x,y,z)$. We denote by $g_c$ the pseudo-Riemannian of $M_c$. We have, for instance,
  	\[ g_c(X_{e_1},X_{e_1})(x,y,z,c)=h(e_1,e_1)(x,y,z,c)=\prec e_1\bullet e_1,(x,y,z,c)\succ=y. \]
  	So, one can see that the matrix of $g_c$ in $(X_{e_1},X_{e_2},X_{e_3})$ is the passage matrix $P$ from  $(X_{e_1},X_{e_2},X_{e_3})$ to $(\partial_x,\partial_y,\partial_z)$ and hence
  	\[ g_c=\frac1{c}\left(2dxdz+dy^2-\frac{2z}{c}dydz+\frac{(z^2-yc)}{c^2}dz^2  \right). \]
  	The signature of this metric is $(+,+,-)$ if $c>0$ and $(+,-,-)$ if $c<0$. One can check easily that $g_c$ is the restriction of $\na d\phi$ to $M_c$, where
  	\[ \phi(x,y,z,t)=\frac{z^4}{12t^3}+\frac{y^2}{2t}-\frac{z^2y}{2t}+\frac{xz}{t}. \]
  
  \end{example}

  \subsection{Multiplicative contravariant pseudo-Hessian structures}
  A contravariant pseudo-Hessian structure $(\na,h)$ on a Lie group $G$ is called multiplicative if the multiplication $m:(G\times G,\na\oplus\na,h\oplus h)\too (G,\na,h)$ is affine and sends $h\oplus h$ to $h$.
  
  \begin{lemma}\label{le1} Let $G$ be a connected Lie group and $\na$ a connection on $G$ such that the multiplication $m:(G\times G,\na\oplus\na)\too (G,\na)$ preserves the connections. Then $G$ is abelian and $\na$ is bi-invariant.
  	
  \end{lemma}
  \begin{proof} We will denote by $\chi^r(G)$ (resp. $\chi^l(G)$) the space of left invariant vector fields (resp. the right invariant vector fields) on $G$. It is clear that for any $X\in \chi^r(G)$ and $Y\in \chi^l(G)$, the vector field $(X,Y)$ on $G\times G$ is $m$-related to the vector field $X+Y$ on $G$: 
  	$$Tm(X_a,Y_b)=X_a. b+a. Y_b=X_{ab}+Y_{ab}=(X+Y)_{ab}
  	$$ 
  	It follows that for any $X_1,X_2\in \chi^r(G)$ and $Y_1,Y_2\in \chi^l(G)$, the vector field $(\nabla\oplus\nabla)_{(X_1,Y_1)}(X_2,Y_2)$ is $m$-related to $\nabla_{X_1+Y_1}(X_2+Y_2)$, hence:
  	$$ Tm((\nabla_{X_1}X_2)_a,(\nabla_{Y_1}Y_2)_b)=(\nabla_{(X_1+Y_1)}(X_2+Y_2))_{ab}
  	$$
  	So we get 
  	\begin{equation}\label{equ}
  	(\nabla_{X_1}X_2)_a.b+a.(\nabla_{Y_1}Y_2)_b=(\nabla_{X_1}X_2+\nabla_{X_1}Y_2+\nabla_{Y_1}X_2+\nabla_{Y_1}Y_2)_{ab}
  	\end{equation}
  	If we take $Y_1=0=Y_2$ we obtain that $\nabla$ is right invariant.  In the same way we get that $\nabla$ is left invariant. Now, if we return back to the equation \ref{equ} we obtain that for any $X\in \chi^r(G)$ and $Y\in \chi^l(G)$ we have $\nabla X=0=\nabla Y$. This implies that any left invariant vector field is also right invariant ; indeed, if $Y=\sum_{i=1}^{n}f_iX_i$ with $Y\in \chi^l(G)$ and $X_i\in \chi^r(G)$ then 
  	$X_jf_i=0$ for all $i,j=1,\cdot,n$. Hence the adjoint representation is trivial and hence $G$ must be abelian.
  	\end{proof}
  	
  	At the end of the paper, we give another proof of this Lemma based on parallel transport.
  \begin{corollary} Let $(\na,h)$ be multiplicative contravariant pseudo-Hessian structure on a simply connected Lie group $G$. Then $G$ is a vector space, $\na$ its canonical affine connection and $h$ is linear.
  \end{corollary}

\begin{example}
  Based on the classification of complex associative commutative algebras given in \cite{Ra}, we can give a list of examples of affine contravariant pseudo-Hessian structures up to dimension 4.  \begin{enumerate}
  	\item On $\mathbb{R}^{2}$:
  	\[
  	h_{1}=
  	\begin{pmatrix}
  	x_{2} & 0 \\
  	0 & 0
  	\end{pmatrix},\;  h_{2}=
  	\begin{pmatrix}
  	x_{1} & x_{2} \\
  	x_{2} & 0
  	\end{pmatrix}\esp h_3=
  	\begin{pmatrix}
  	x_{2} & 1 \\
  	1 & 0
  	\end{pmatrix}.
  	\]
  	
  	\item On $\mathbb{R}^{3}:$ 
  	\begin{eqnarray*}
  	h_{1}&=&
  	\begin{pmatrix}
  	a & 0 & x_{2}\\
  	0 & 0 & 0 \\
  	x_{2} & 0 & b
  	\end{pmatrix},\;\;
  	h_{2}=
  	\begin{pmatrix}
  	x_{2} & x_{3} & a\\
  	x_{3} & a & 0 \\
  	a & 0 & 0
  	\end{pmatrix},\;\;
  	h_{3}=
  	\begin{pmatrix}
  	a & 0 & x_{1}\\
  	0 & 0 & x_{2} \\
  	x_{1} & x_{2} & x_{3}
  	\end{pmatrix},\\
  	h_{4}&=&
  	\begin{pmatrix}
  	x_{2} & 0 & x_{2}\\
  	0 & 0 & x_{2}+a \\
  	x_{2} & x_{2}+a & x_{3}
  	\end{pmatrix}\esp h_{5}=
  	\begin{pmatrix}
  	x_{2} & 0 & x_{1}\\
  	0 & 0 & x_{2} \\
  	x_{1} & x_{2} & x_{3}
  	\end{pmatrix}.
  	\end{eqnarray*}
  	\item On $\mathbb{R}^{4}:$
  	\begin{eqnarray*}
  	h_{1}&=&
  	\begin{pmatrix}
  	x_{3} & a & x_{4}+b & 0\\
  	a & -x_{4}+c & 0 & 0\\
  	x_{4}+b & 0 & 0 & 0\\
  	0 & 0 & 0 & 0
  	\end{pmatrix},\;\;
  	h_{2}=\begin{pmatrix}
  	x_{2} & x_{3} & x_{4} & a\\
  	x_{3} & x_{4} & a & 0\\
  	x_{4} & a & 0 & 0\\
  	a & 0 & 0 & 0
  	\end{pmatrix},\;\;
  	h_{3}=
  	\begin{pmatrix}
  	x_{1} & x_{2} & x_{3} & x_{4}\\
  	x_{2} & 0 & 0 & 0\\
  	x_{3} & 0 & 0 & 0\\
  	x_{4} & 0 & 0 & 0
  	\end{pmatrix},\\
  	h_{4}&=&
  	\begin{pmatrix}
  	x_{1} & x_{2} & x_{3} & x_{4}\\
  	x_{2} & x_{4} & 0 & 0\\
  	x_{3} & 0 & 0 & 0\\
  	x_{4} & 0 & 0 & 0
  	\end{pmatrix}\esp 
  	h_{5}=
  	\begin{pmatrix}
  	x_{1} & x_{2} & x_{3} & x_{4}\\
  	x_{2} & x_{3} & x_{4} & 0\\
  	x_{3} & x_{4} & 0 & 0\\
  	x_{4} & 0 & 0 & 0
  	\end{pmatrix}.
  	\end{eqnarray*}
  	
  \end{enumerate}
  \end{example}

  \section{Quadratic contravariant pseudo-Hessian structures}\label{section5}

  Let $V$ be a vector space of dimension $n$. Denote by $\na$ its canonical affine connection. A  symmetric bivector field $h$ on $V$ is  quadratic if there exists a basis $\B$ of $V$ such that, for any $i,j=1,\ldots,n$,
  \[ h(dx_i,dx_j)=\sum_{l,k=1}^na_{l,k}^{i,j}x_lx_k, \]where the $a_{k,l}^{i,j}$ are real constants and $(x_1,\ldots,x_n)$ are the linear coordinates associated to $\B$.  
  
  For any linear endomorphism $A$ on $V$ we denote by $\wi A$ the associated linear vector field on $V$.
  
  The key point is that if $h$ is a quadratic contravariant pseudo-Hessian bivector field on $V$ then its divergence is a linear vector field, i.e., $\dev_\na(h)=\wi{L^h}$ where $L^h$ is a linear endomorphism of $V$.  Moreover, if $F=(A,u)$ is an affine transformation of $V$ then $\dev_\na(F_*h)=\wi{A^{-1}L^hA}$. So the Jordan form of $L_h$ is an invariant of the quadratic contravariant pseudo-Hessian structure. By using Maple we can classify quadratic contravariant pseudo-Hessian structures on $\R^2$. The same approach has been used by \cite{du} to classify quadratic Poisson structures on $\R^4$. Note that if $h$ is a quadratic  contravariant pseudo-Hessian tensor on $\R^n$ then its  associated Poisson tensor on $\C^n$ is also quadratic.
  
  \begin{theorem}\begin{enumerate}
  		\item Up to an affine isomorphism, there is two quadratic contravariant pseudo-Hessian structures on $\R^2$ which are divergence free
  		\[ h_1=\left(\begin{array}{cc}0&0\\0&ux^2
  		\end{array}  \right) \esp h_2=\left(\begin{array}{cc}\frac{r^2x^2}{c}-2rxy+cy^2&\frac{r^3x^2}{c^2}
  		-\frac{2r^2xy}c+ry^2\\\frac{r^3x^2}{c^2}-\frac{2r^2xy}c+ry^2&-\frac{2r^3xy}{c^2}+\frac{r^4x^2}{c^3}+\frac{r^2y^2}{c}
  		\end{array}  \right).\]
  		\item Up to an affine isomorphism, there is two quadratic contravariant pseudo-Hessian structures on $\R^2$ with the divergence equivalent to the Jordan form $\left( \begin{array}{cc}
  		a&1\\0&a
  		\end{array} \right)$,
  		\[ h_1=\left(\begin{array}{cc}cy^2+xy&0\\0&0
  		\end{array}  \right) \esp h_2=\left(\begin{array}{cc}\frac12xy+cy^2&\frac{y^2}4\\\frac{y^2}4&0
  		\end{array}  \right).\]
  		\item Up to an affine isomorphism, there is five quadratic contravariant pseudo-Hessian structures on $\R^2$ with diagonalizable divergence
  		\[ h_1=\left(\begin{array}{cc}ax^2&0\\0&by^2
  		\end{array}  \right),\;h_2=\left(\begin{array}{cc}ax^2+by^2&0\\0&0
  		\end{array}  \right),\;h_3=\left(\begin{array}{cc}ax^2&axy\\axy&ay^2
  		\end{array}  \right), \]
  		\[ h_4=\left(\begin{array}{cc}\frac{2r^2x^2}c-2rxy+cy^2&ry^2\\ry^2&\frac{2r^2y^2}c
  		\end{array}  \right)\esp h_5=\left(\begin{array}{cc}(\frac{2p^2}u+\frac{q}2)x^2+\frac{pqxy}u+\frac{q^2y^2}{4u}&px^2+qxy-\frac{pqy^2}{2u}\\px^2+qxy-\frac{pqy^2}{2u}&(\frac{2p^2}u+\frac{q}2)y^2+ux^2-2pxy
  		\end{array}  \right). \]
  		\item Up to an affine isomorphism, there is a unique quadratic pseudo-Hessian structure on $\R^2$ with the divergence having non real eigenvalues
  		\[  h=\left(\begin{array}{cc}-2pxy-ux^2+uy^2&px^2-py^2-2uxy\\px^2-py^2-2uxy&2pxy+ux^2-uy^2
  		\end{array}  \right). \]
  	\end{enumerate}
  	
  \end{theorem}
  
  \begin{example} The study of quadratic contravariant pseudo-Hessian structures on $\R^3$ is more complicated and we give here a
  class of quadratic pseudo-Hessian structures on $\R^3$ of the form $\wi A\odot \wi{\mathrm{I}}_3$ where $\wi A$ is linear.
  \begin{enumerate}
  	\item 
   $A$ is diagonal:
  \[ h_1=\left( \begin{array}{ccc}x^2&xy&xz\\xy&y^2&yz\\xz&yz&z^2
  \end{array}  \right)\esp h_2=\left( \begin{array}{ccc}x^2&xy&0\\xy&y^2&0\\0&0&-z^2
  \end{array}  \right). \]
  
  \item  $A=\left(\begin{array}{ccc}a&1&0\\0&a&0\\0&0&b
  \end{array} \right)$:
  	\begin{eqnarray*} h_3&=&\left(\begin{array}{ccc}2x(y-px)&(y-px)y+pyx&pxz+(y-px)z\\
  	(y-px)y+pyx&2py^2&2pyz\\pxz+(y-px)z&2pyz&2pz^2
  	\end{array} \right),\\ h_4&=&\left(\begin{array}{ccc}2x(y+px)&(y+px)y-pyx&pxz+(y+px)z\\
  	(y+px)y+pyx&-2py^2&0\\pxz+(y+px)z&0&2pz^2
  	\end{array} \right). \end{eqnarray*}
  \end{enumerate}
  \end{example}
  	
  	\section{Right-invariant contravariant pseudo-Hessian structures on Lie groups}\label{section6}

  	Let $(\mathfrak{g},\bullet)$ be a left symmetric algebra, i.e., for any $u,,v,w\in\G$,
  	\[ \mathrm{ass}(u,v,w)=\mathrm{ass}(v,u,w)\esp \mathrm{ass}(u,v,w)=(u\bullet v)\bullet w-u\bullet(v\bullet w). \]
  	This implies that $[u,v]=u\bullet v-v\bullet u$ is a Lie bracket on $\G$ and
  	 $\mathrm{L}:\mathfrak{g}\rightarrow\mathrm{End}(\mathfrak{g})$, $u\mapsto \mathrm{L}_u$ is a representation of the  Lie algebra $(\G,\br)$. We denote by $\mathrm{L}_u$ the left multiplication by $u$.
  	 
  	  We consider a connected Lie group $G$ whose Lie algebra is $(\mathfrak{g},\br)$ and we define on $G$ a right invariant connection by 
  	\begin{equation}\label{eq}
  	\nabla_{u^{-}}v^{-}=-(u\bullet v)^{-},
  	\end{equation}
  	where $u^{-}$ is the right vector field associated to $u\in\mathfrak{g}$. This connection is torsionless and without curvature and hence $(G,\na)$ is an affine manifold. Let $r\in\G\otimes\G$ which is symmetric and let $r^-$ be the associated right invariant symmetric bivector field.
  	
  	\begin{proposition} $(G,\na,r^-)$ is a contravariant pseudo-Hessian manifold if and only if, for any $\al,\be,\ga\in\G^*$,
  		\begin{equation}\label{hessian} [[r,r]](\al,\be,\ga):= \prec\ga,r_\#([\al,\be]_r)-[r_\#(\al),r_\#(\be)]\succ=0, \end{equation}where
  		\[ [\al,\be]_r=\mathrm{L}_{r_\#(\al)}^*\be-\mathrm{L}_{r_\#(\be)}^*\al\esp \prec \mathrm{L}_{u}^*\al,v\succ=-\prec\al,u\bullet v\succ. \]
  		In this case,  the product on $\G$ given by $\al.\be=\mathrm{L}_{r_\#(\al)}^*\be$ is left symmetric, $\br_r$ is a Lie bracket  and $r_\#$ is a morphism of Lie algebras.
  		
  	\end{proposition}
  	\begin{proof} Note first that for any $\al\in\G^*$, $r^-_\#(\al^-)=(r_\#(\al))^-$ and $\na_{u^-}\al^-=-\left(\mathrm{L}_u^*\al \right)^-$ and hence, for any $\al,\be,\ga\in\G^*$,
  	\[ \na_{r^-_\#(\al^-)}(r^-)(\be^-,\ga^-)=r(\mathrm{L}_{r_\#(\al)}^*\be,\ga)+r(\be, \mathrm{L}_{r_\#(\al)}^*\ga). \]So, $(G,\na,r^-)$ is a contravariant pseudo-Hessian manifold if and only if, for any $\al,\be,\ga\in\G^*$,
  	\begin{eqnarray*}
  	0&=&r(\mathrm{L}_{r_\#(\al)}^*\be,\ga)+r(\be, \mathrm{L}_{r_\#(\al)}^*\ga)
  	-r(\mathrm{L}_{r_\#(\be)}^*\al,\ga)-r(\al, \mathrm{L}_{r_\#(\be)}^*\ga)\\
  	&=&\prec\ga,r_\#([\al,\be]_r)-r_\#(\al)\bullet r_\#(\be)+r_\#(\be)\bullet r_\#(\al)\succ\\
  	&=&\prec\ga,r_\#([\al,\be]_r)-[r_\#(\al), r_\#(\be)]\succ
  	\end{eqnarray*}and the first part of the proposition follows. Suppose now that $r_\#([\al,\be]_r)=[r_\#(\al), r_\#(\be)]$ for any $\al,\be\in\G^*$. Then, for any $\al,\be,\ga\in\G^*$,
  	\begin{eqnarray*}
  	\mathrm{ass}(\al,\be,\ga)-\mathrm{ass}(\be,\al,\ga)&=&\mathrm{L}_{r_\#([\al,\be]_r)}^*\ga-\mathrm{L}_{r_\#(\al)}^*\mathrm{L}_{r_\#(\be)}^*\ga+\mathrm{L}_{r_\#(\be)}^*\mathrm{L}_{r_\#(\al)}^*\ga=0.
  	\end{eqnarray*}This completes the proof.
  		\end{proof}
  		
  		\begin{definition}\begin{enumerate}\item
  			 Let $(\G,\bullet)$ be a left symmetric algebra. A symmetric bivector $r\in\G\otimes\G$ satisfying $[[r,r]]=0$ is called  a $S$-matrix.
  			
  		\item	A left symmetric algebra $(\G,\bullet,r)$ endowed with a $S$-matrix is called a contravariant pseudo-Hessian algebra.
  			\end{enumerate}
  		\end{definition}
  		
  		Let $(\G,\bullet,r)$ be a contravariant pseudo-Hessian algebra, $[u,v]=u\bullet v-v\bullet u$ and $G$ a connected Lie group  with $(\G,\br)$ as a Lie algebra.
  	We have shown that $G$ carries a right invariant contravariant pseudo-Hessian structure $(\na,r^-)$. On the other hand, in Section \ref{section3}, we have associated to $(\na,r^-)$ a flat connection $\overline{\na}$, a complex structure $J$ and a Poisson tensor $\Pi$ on $TG$. Now we will show  that $TG$ carries a structure of Lie group and the triple $(\overline{\na},J,\Pi)$ is right invariant. This structure of Lie group on $TG$ is different from the usual one defined by the adjoint action of $G$ on $\G$.

  Let us start  with a general algebraic construction which is interesting on its own.	Let $(\G,\bullet)$ be a left symmetric algebra, put
   $\Phi(\G)=\G\times\G$ and define a product $\star$ and  a  bracket on $\Phi(\G)$ by
  \[ (a,b)\star(c,d)=(a\bullet c,a\bullet d)\esp  [(a,b),(c,d)]=([a,c],a\bullet d-c\bullet b), \]for any $(a,b),(c,d)\in\Phi(\G)$.
  It is easy to check that $\star$ is left symmetric, $\br$ is the commutator of $\star$ and hence is a Lie bracket. We define also $J_0:\Phi(\G)\too\Phi(\G)$ by $J_0(a,b)=(b,-a)$. It is also a straightforward computation to check that
  \[ N_{J_0}((a,b),(c,d))=[J_0(a,b),J_0(c,d)]-J_0[(a,b),J_0(c,d)]-J_0[J_0(a,b),(c,d)]-
  [(a,b),(c,d)]=0. \]

For $r\in\otimes^2\G$ symmetric, we define $R\in\otimes^2\Phi(\G)$  by
  	\begin{equation}\label{R} R((\al_1,\be_1),(\al_2,\be_2))=r(\al_1,\be_2)-r(\al_2,\be_1),
  	 \end{equation}for any $\al_1,\al_2,\be_1,\be_2\in\G^*$.    We have obviously that $R_\#(\al_1,\be_1)=(-r_\#(\be_1),r_\#(\al_1))$.
  	\begin{proposition}  $[[r,r]]=0$ if and only if $[R,R]=0$, where $[R,R]$ is the Schouten bracket associated to the Lie algebra structure of $\Phi(\G)$ and given by
  		\[ [R,R](\al,\be,\ga)=\oint_{\al,\be,\ga}\prec\ga,[R_\#(\al),R_\#(\be)]\succ,\quad \al,\be,\ga\in\Phi^*(\G). \]

  	\end{proposition}	
  	\begin{proof} For any $\al=(\al_1,\al_2),\be=(\be_1,\be_2),\ga=(\ga_1,\ga_2)\in\Phi(\G)^*$,
  	\begin{eqnarray*}
  	\prec\ga,[R_\#(\al),R_\#(\be)]&=&
  	\prec (\ga_1,\ga_2),[(-r_\#(\al_2),r_\#(\al_1)),(-r_\#(\be_2),r_\#(\be_1))]\succ\\
  	&=&\prec\ga_1,[r_\#(\al_2),r_\#(\be_2)]\succ-\prec\ga_2,r_\#(\al_2)\bullet r_\#(\be_1)\succ+\prec\ga_2,r_\#(\be_2)\bullet r_\#(\al_1)\succ\\
  	&=&\prec\ga_1,[r_\#(\al_2),r_\#(\be_2)]\succ+\prec\be_1,r_\#(\mathrm{L}_{r_\#(\al_2)}^*\ga_2)\succ-\prec\al_1,r_\#(\mathrm{L}_{r_\#(\be_2)}^*\ga_2)\succ,\\
  	\prec\be,[R_\#(\ga),R_\#(\al)]\succ&=&
  	\prec\be_1,[r_\#(\ga_2),r_\#(\al_2)]\succ+\prec\al_1,r_\#(\mathrm{L}_{r_\#(\ga_2)}^*\be_2)\succ-\prec\ga_1,r_\#(\mathrm{L}_{r_\#(\al_2)}^*\be_2)\succ\\
  	\prec\al,[R_\#(\be),R_\#(\ga)]\succ&=&
  	\prec\al_1,[r_\#(\be_2),r_\#(\ga_2)]\succ+\prec\ga_1,r_\#(\mathrm{L}_{r_\#(\be_2)}^*\al_2)\succ-\prec\be_1,r_\#(\mathrm{L}_{r_\#(\ga_2)}^*\al_2)\succ.
  	\end{eqnarray*}	So
  	\[ [R,R](\al,\be,\ga)=-
  	[[r,r]](\be_2,\ga_2,\al_1)-[[r,r]](\ga_2,\al_2,\be_1)-[[r,r]](\al_2,\be_2,\ga_1)
  	 \]and the result follows.
  			\end{proof}

  Let $G$ be a Lie group whose Lie algebra is $(\G,\br)$ and 	let $\rho:G\too\mathrm{GL}(\G)$ be the homomorphism of groups such that $d_e\rho=\mathrm{L}$ where $\mathrm{L}:\G\too \mathrm{End}(\G)$ is the representation associated to $\bullet$. 
  	Then the  product
  	\[ (g,u).(h,v)=(gh,u+\rho(g)(v)),\quad g,h\in G,u,v\in\G \]induces a Lie group structure on $G\times\G$ whose Lie algebra is $(\Phi(\G),\br)$. The complex endomorphism $J_0$ and the left symmetric product $\star$ induce a right invariant complex tensor $J_0^-$ and a a right invariant connection $\wi{\na}$ given by
  	\[ J_0^-(a,b)^-=(b,-a)^-\esp \wi{\na}_{(a,b)^-}(c,d)^-=-((a,b)\star(c,d))^-. \]
  	
  	Let $r\in\otimes^2\G$ symmetric such that $[[r,r]]=0$, $r^-$  the associated right invariant symmetric bivector field and $\na$ the affine connection given by \eqref{eq}. Then $(G,\na,r^-)$ is a contravariant pseudo-Hessian manifold and let $\overline{\na}$, $J$ and $\Pi$ be  the associated structure on $TG$ defined in Section \ref{section3}.

  	\begin{theorem}\label{main} If we identify $TG$ with $G\times\G$ by $u_g\too (g,T_gR_{g^{-1}}u_g)$ we denote also by $\Pi$,  $\overline{\na}$ and $J$ the images of $\Pi$,  $\overline{\na}$ and $J$ under this identification then $\Pi=R^-$,  $\overline{\na}=\wi{\na}$ and $J=J_0^-$.
  		
  	\end{theorem}

  	To prove this theorem, we need some preparation.
  	\begin{proposition}\label{par} Let $(G,\na)$ be a Lie group endowed with a right invariant connection
  		 and $\ga:[0,1]\too G$ a curve. Let  $V:[0,1]\too TG$ be a vector field along $\ga$.  We define $\mu:[0,1]\too \G$ and $W:[0,1]\too\G$ by
  		\[ \mu(t)=T_{\ga(t)}R_{\ga(t)^{-1}}(\ga'(t))\esp W(t)=T_{\ga(t)}R_{\ga(t)^{-1}}(V(t)). \]
  		Then $V$ is parallel along $\ga$ with respect $\na$ if and only if
  		\[ W'(t)-\mu(t)\bullet W(t)=0, \]
  		where $u\bullet v=-(\na_{u^-}v^-)(e)$.
  		\end{proposition}
  	
  	\begin{proof} We consider $(u_1,\ldots,u_n)$ a basis of $\G$ and $(X_1,\ldots,X_n)$ the corresponding right invariant vector fields. Then
  		\[ \begin{cases}\di
  		\mu(t)=\sum_{i=1}^n\mu_i(t)u_i,\;W(t)=\sum_{i=1}^nW_i(t)u_i,\\
  		\di \ga'(t)=\sum_{i=1}^n\mu_i(t)X_i,\;V(t)=\sum_{i=1}^nW_i(t)X_i.
  		\end{cases} \]	Then
  		\begin{eqnarray*}
  			\na_{t}V(t)&=&\sum_{i=1}^nW_i'(t)X_i+
  			\sum_{i=1}^nW_i(t)\na_{\ga'(t)}X_i\\
  			&=&\sum_{i=1}^nW_i'(t)X_i+\sum_{i,j=1}^nW_i(t)\mu_j(t)\na_{X_j}X_i\\
  			&=&\sum_{i=1}^nW_i'(t)X_i-\sum_{i,j=1}^nW_i(t)\mu_j(t)(u_j\bullet u_i)^-\\
  			&=&\left(W'(t)-\mu(t)\bullet W(t)\right)^-
  		\end{eqnarray*}and the result follows having in mind that $u^-$ is the right invariant vector field associated to $u\in\G$.
  	\end{proof}
  	
  	Let $(G,\na)$ be a Lie group endowed with a right invariant connection. Then $\na$ induces a splitting of $TTG=\ker dp\oplus \mathcal{H}$. For any tangent vector  $X\in T_gG$, we denote by $X^v,X^h\in T_{(g,u)}TG$ the vertical and the horizontal lift of $X$.

  	\begin{proposition} If we identify $TG$ to $G\times\G$ by $X_g\mapsto (g,T_gR_{g^{-1}}(X_g))$ then for any $X\in T_gG$,
  		\[ X^v(g,u)=(0,T_gR_{g^{-1}}(X))\esp X^h(g,u)=(X,T_gR_{g^{-1}}(X)\bullet u). \]
  		
  	\end{proposition}
  	\begin{proof} The first relation is obvious. Recall that the horizontal lift of $X$ at $u_g\in TG$ is given by:
  		\[ X^h(u_g)=\frac{d}{dt}_{|t=0}V(t) \]where $V:[0,1]\too TG$ is the parallel vector field along
  		$\ga:[0,1]\too G$ a curve such that $\ga(0)=g$ and $\ga'(0)=X$. If we denote by $\Theta_R:TG\too G\times\G$ the identification $u_g\mapsto(g,T_gR_{g^{-1}}(u_g))$ then, by virtue of Proposition \ref{par},
  		\[ T_{u_g}\Theta_R(X^h)=\frac{d}{dt}_{|t=0}(\ga(t),W(t))=(X,T_gR_{g^{-1}}(X)\bullet u). \]
  		
  	\end{proof}

  	We consider now a left symmetric algebra $(\G,\bullet)$, $G$ a connected Lie group  associated to $(\G,\br)$, $\na$ the right invariant affine connection associated to $\bullet$. We have seen that $G\times\G$ has a structure of Lie group. We identify $TG$ to $G\times \G$ and, for any vector field $X$ on $G$, we denote by $X^v$ and $X^h$ the vector fields on $G\times\G$ obtained by the identification from the horizontal and the vertical lift of $X$. For $a,b\in\G$, $\al,\be\in \G^*$, $a^-$ (resp. $\al^-$) is the right invariant vector field (resp. 1-form) on $G$ associated to $a$ (resp. $\al$), $(a,b)^-$ (resp. $(\al,\be)^-$) the right invariant vector field (resp. 1-form) on $G\times\G$ associated to $(a,b)$ (resp. $(\al,\be)$). 
  	\begin{proposition} For any $(a,b)\in\G\times\G$ and $(\al,\be)\in\G^*\times\G^*$,
  		\[ (a,b)^-=(a^-)^h+(b^-)^v\esp (\al,\be)^-=(\al^-)^h+(\be^-)^v. \]
  		
  	\end{proposition}
  	
  	\begin{proof} We have
  		\begin{eqnarray*} (a,b)^-(g,u)&=&T_{(e,0)}R_{(g,u)}(a,b)\\&=&\frac{d}{dt}_{|t=0}(\exp(ta),tb)(g,u)\\&=& 
  			\frac{d}{dt}_{|t=0}(\exp(ta)g,tb+\rho(\exp(ta))(u))\\&=&(a^-(g),b+a\bullet u)\\
  			&=&(a^-(g),T_gR_{g^{-1}}(a^-(g))\bullet u)+(0,T_gR_{g^{-1}}(b^-(g))\\
  			&=&(a^-)^h(g,u)+(b^-)^v(g,u).\quad\quad \mbox{(Proposition \ref{par})}
  		\end{eqnarray*}The second relation can be deduced easily from the first one.
  		
  	\end{proof}

  	\paragraph{Proof of Theorem \ref{main} }\begin{proof} Let $\Pi$ be the Poisson tensor on $G\times\G$ associated to $r^-$. Then,  by using the precedent proposition,
  	\begin{eqnarray*}
  	\Pi((\al_1,\be_1)^-,(\al_2,\be_2)^-)&=&\Pi((\al_1^-)^h+(\be_1^-)^v,(\al_2^-)^h+(\be_2^-)^v)\\
  	&=&r^-(\al_1^-,\be_2^-)-r^-(\al_2^-,\be_1^-)\\
  	&=&r(\al_1,\be_2)-r(\al_2,\be_1)\\
  	&=&R^-((\al_1,\be_1)^-,(\al_2,\be_2)^-).
  	\end{eqnarray*}In the same way,
  	\begin{eqnarray*}
  		J_0^-(a,b)^-&=&(b,-a)^-=(b^-)^h-(a^-)^v,\\
  		J(a,b)^-&=&(b^-)^h-(a^-)^v,\\
  		\overline{\na}_{(a,b)^-}(c,d)^-&=&(\na_{a^-}c^-)^h+(\na_{a^-}d^-)^v=
  		-((a\bullet c)^-)^h-((a\bullet d)^-)^v=-((a,b).(c,d))^-=\wi\na_{(a,b)^-}(c,d)^-.
  	\end{eqnarray*}\end{proof}
  	
  	Let $(\G,\bullet)$ be a left symmetric algebra, $(M,\na)$ and affine manifold and $\rho:\G\too \Ga(TM)$ a linear map such that $\rho(u\bullet v)=\na_{\rho(u)}\rho(v)$. Then $\rho$ defines an action on $M$ of the Lie algebra $(\G,\br)$. We consider $\rho^l:\Phi(\G)\too\Ga(TTM)$, $(u,v)\too \rho(u)^h+\rho^v(v)$. It is easy to check that
  	\[ \rho^l([a,b])=[\rho^l(a),\rho^l(b)]. \]
  	Let $r\in\otimes^2\G$ satisfying $[[r,r]]=0$ and $R\in\otimes^2\Phi(\G)$ given by \eqref{R}. 
  	\begin{theorem}\label{action} The bivector field on $TM$ associated to $\rho(r)$ is $\rho^l(R)$ which is a Poisson tensor and $(M,\na,\rho(r))$ is a  contravariant pseudo-Hessian manifold.
  		
  	\end{theorem}
  	
  	\begin{proof} Let $(e_1,\ldots,e_n)$ a basis of $\G$ and $E_{i}=(e_i,0)$ and $F_i=(0,e_i)$. Then $(E_1,\ldots,E_n,F_1,\ldots,F_n)$ is a basis of $\Phi(\G)$. Then
  		\[ r=\sum_{i,j}r_{i,j}e_i\otimes e_j\esp R=\sum_{i,j} r_{i,j}\left( E_i\otimes F_j-F_i\otimes E_j\right). \]So
  		\[ \rho(r)=\sum_{i,j=1}^nr_{i,j}\rho(e_i)\otimes \rho(e_j)\esp \rho^l(R)
  		=\sum_{i,j=1}^n r_{i,j}\left( 
  		\rho(e_i)^h\otimes \rho(e_j)^v-\rho(e_i)^v\otimes \rho(e_j)^h\right). \]Then for any $\al,\be\in\Om^1(M)$
  		\[ \rho^l(R)(\al^v,\be^v)=\rho^l(R)(\al^h,\be^h)=0\esp \rho^l(R)(\al^h,\be^v)=
  		\rho(r)(\al,\be)\circ p. \]	
  		According to Proposition \ref{R}, $R$ is a solution of the classical Yang-Baxer equation and hence $\rho^l(R)$ is a Poisson tensor. By using Theorem \ref{poisson}, we get that $(M,\na,\rho(r))$ is a  contravariant pseudo-Hessian manifold.
  	\end{proof}
  	
  	\begin{example}\begin{enumerate}\item
  		
  		 Let $\G=\mathrm{gl}(n,\R)$ be the Lie algebra of $n$-square matrices. It is has a structure of left symmetric algebra given by $A\bullet B=BA$. Let $\rho:\G\too\Ga(T\R^n)$ given by $\rho(A)=A$. Then $\rho(A\bullet B)=\na_AB$, where $\na$ is the canonical connection of $\R$. According to Theorem \ref{action}, any $S$-matrix on $\G$  gives rise to a quadratic contravariant pseudo-Hessian structure on $\R^n$. 
  		 
  		 \item More generally, let $(M,\na)$ be an affine manifold and $\G$ the finite dimensional Lie algebra of affine vector fields. Recall that $X\in\G$ if for any $Y,Z\in\Ga(TM)$,
  		 \[ [X,\na_YZ]=\na_{[X,Y]}Z+\na_Y[X,Z]. \]
  		 Since the curvature and the torsion of $\na$ vanish this is equivalent to
  		 \[ \na_{\na_YZ}X=\na_Y\na_ZX. \]
  		 From this relation, one can see easily that, for any $X,Y\in\G$, $X\bullet Y:=\na_XY\in\G$ and $(\G,\bullet)$ is an associative finite dimensional Lie algebra which acts on $M$ by $\rho(X)=X$. Moreover, $\rho(X\bullet Y)=\na_XY$. According to Theorem \ref{action}, any $S$-matrix on $\G$  gives rise to a  contravariant pseudo-Hessian structure on $M$. 
  		 \end{enumerate}   
  		
  	\end{example}

  	\subsection*{\textbf{Classification of two-dimensional contravariant pseudo-Hessian algebras}}
  	
  	Using the classification of two-dimensional non-abelian left symmetric algebras given in \cite{Bu} and the classification of abelian left symmetric algebras given in \cite{Ra}, we give a classification (over the field $\mathbb{R}$ ) of 
  	2-dimensional contravariant pseudo-Hessian  algebras. We proceed as follows:
  	\begin{enumerate}
  		\item For any left symmetric 2-dimensional algebra $\G$, we determine its automorphism group $\mathrm{Aut}(\G)$ and the space of $S$-matrices on $\G$, we denote by $\mathcal{A}(\G)$.
  		\item We give the quotient $\mathcal{A}(\G)/\sim$ where $\sim$ is the equivalence relation:
  		
  \begin{equation*}
  r^1\sim r^{2}\Longleftrightarrow \exists\; A\in\mathrm{Aut}(\mathfrak{g})\; { or }\; \exists\;\lambda\in\mathbb{R} \text{ such that } r_{\sharp}^{2}=A\circ r_{\sharp}^{1}\circ A^{t} \text{ or } r^{2}=\lambda r^{1}.
  \end{equation*}
  		
  	\end{enumerate}

  	\begin{center}
  		\begin{tabular}{|p{4cm}|p{3cm}|p{6cm}|}
  			\hline $(\mathfrak{g},.)$ & $\mathrm{Aut}(\mathfrak{g})$ & $\displaystyle {\mathcal{A}(\G)}/{\sim}$ \\ 
  			\hline $b_{1,\alpha\neq-1,1}\newline e_{2}.e_{1}=e_{1}, e_{2}.e_{2}=\alpha e_{2}$ &  	$\displaystyle\begin{pmatrix}
  			a & 0  \\
  			0 & 1
  			\end{pmatrix}, a\neq0$  &  	$\displaystyle r_{\sharp}^{1}=\begin{pmatrix}
  			1 & 0  \\
  			0 & 0
  			\end{pmatrix}; r_{\sharp}^{2}=\begin{pmatrix}
  			0 & 0  \\
  			0 & 1
  			\end{pmatrix};r_{\sharp}^{3}=0$ \\  
  			\hline $b_{1,\alpha=-1}\newline e_{2}.e_{1}=e_{1}, e_{2}.e_{2}=-e_{2}$ &  	$\begin{pmatrix}
  			a & 0  \\
  			0 & 1
  			\end{pmatrix},a\neq0$  &  $r_{\sharp}^{1}=\begin{pmatrix}
  			b & 1  \\
  			1 & 0
  			\end{pmatrix};	r_{\sharp}^{2}=\begin{pmatrix}
  			1 & 0  \\
  			0 & 0
  			\end{pmatrix};\newline r_{\sharp}^{3}=\begin{pmatrix}
  			0 & 0  \\
  			0 & 1
  			\end{pmatrix}; r_{\sharp}^{4}=0$ \\
  			\hline $b_{1,\alpha=1}\newline e_{2}.e_{1}=e_{1}, e_{2}.e_{2}=e_{2}$ &  	$\begin{pmatrix}
  			a & 0  \\
  			0 & b
  			\end{pmatrix},ab\neq0$  &  $r_{\sharp}^{1}=\begin{pmatrix}
  			1 & c  \\
  			c & c^{2}
  			\end{pmatrix};	r_{\sharp}^{2}=\begin{pmatrix}
  			0 & 0  \\
  			0 & 1
  			\end{pmatrix};r_{\sharp}^{3}=0$ \\
  			\hline $b_{2,\beta\neq0,1,2}\newline e_{1}.e_{2}=\beta e_{1}, e_{2}.e_{1}=(\beta-1)e_{1},e_{2}.e_{2}=\beta e_{2}$ &  	$\begin{pmatrix}
  			a & b  \\
  			0 & 1
  			\end{pmatrix},a\neq0$  &  $r_{\sharp}^{1}=\begin{pmatrix}
  			1 & 0  \\
  			0 & 0
  			\end{pmatrix};r_{\sharp}^{2}=\begin{pmatrix}
  			0 & 0  \\
  			0 & 1
  			\end{pmatrix}; r_{\sharp}^{3}=0$ \\
  			\hline $b_{2,\beta=1}\newline e_{1}.e_{2}=e_{1},e_{2}.e_{2}=e_{2}$ &  	$\begin{pmatrix}
  			a & b  \\
  			0 & 1
  			\end{pmatrix},a\neq0$  &  $r_{\sharp}^{1}=\begin{pmatrix}
  			1 & c  \\
  			c & c^{2}
  			\end{pmatrix};r_{\sharp}^{2}=\begin{pmatrix}
  			0 & 0  \\
  			0 & 1
  			\end{pmatrix}; r_{\sharp}^{3}=0$ \\
  			\hline $b_{2,\beta=2}\newline e_{1}.e_{2}=2e_{1},\newline e_{2}.e_{1}=e_{1},e_{2}.e_{2}=2e_{2}$ &  	$\begin{pmatrix}
  			a & b  \\
  			0 & 1
  			\end{pmatrix},a\neq0$  &  $r_{\sharp}^{1}=\begin{pmatrix}
  			1 & 0  \\
  			0 & c
  			\end{pmatrix};r_{\sharp}^{2}=\begin{pmatrix}
  			0 & 0  \\
  			0 & 1
  			\end{pmatrix}; r_{\sharp}^{3}=0$ \\
  			\hline $b_{3}\newline e_{2}.e_{1}=e_{1}, e_{2}.e_{2}=e_{1}+e_{2}$ &  	$\begin{pmatrix}
  			1 & b  \\
  			0 & 1
  			\end{pmatrix}$  &  $r_{\sharp}^{1}=\begin{pmatrix}
  			1/2 & 1  \\
  			1 & 1
  			\end{pmatrix};r_{\sharp}^{2}=\begin{pmatrix}
  			1 & 0  \\
  			0 & 0
  			\end{pmatrix}; r_{\sharp}^{3}=0$ \\
  			\hline $b_{4}\newline e_{1}.e_{1}=2 e_{1},e_{1}.e_{2}=e_{2},\newline e_{2}.e_{2}=e_{1}$ &  	$\begin{pmatrix}
  			1 & 0  \\
  			0 & -1
  			\end{pmatrix};\begin{pmatrix}
  			1 & 0  \\
  			0 & 1
  			\end{pmatrix}$  &  $r_{\sharp}^{1}=\begin{pmatrix}
  			1 & 0  \\
  			0 & 2
  			\end{pmatrix};r_{\sharp}^{2}=\begin{pmatrix}
  			1 & 0  \\
  			0 & 0
  			\end{pmatrix}; r_{\sharp}^{3}=0$ \\
  			\hline $b_{5}\newline e_{1}.e_{2}=e_{1}, e_{2}.e_{2}=e_{1}+e_{2}$ &  	$\begin{pmatrix}
  			1 & b  \\
  			0 & 1
  			\end{pmatrix}$  &  $r_{\sharp}^{1}=\begin{pmatrix}
  			1 & 0  \\
  			0 & 0
  			\end{pmatrix}; r_{\sharp}^{2}=0$ \\
  			\hline $As_{2}^{1}\newline e_{1}.e_{1}=e_{2}$ &  	$\begin{pmatrix}
  			a & 0  \\
  			b & a^{2}
  			\end{pmatrix}, a\neq0$  &  $r_{\sharp}^{1}=\begin{pmatrix}
  			0 & 1  \\
  			1 & 0
  			\end{pmatrix};r_{\sharp}^{2}=\begin{pmatrix}
  			0 & 0  \\
  			0 & 1
  			\end{pmatrix}; r_{\sharp}^{3}=0$ \\
  			\hline 
  		\end{tabular}
  	\end{center}
  	\begin{center}
  		\begin{tabular}{|p{4cm}|p{3cm}|p{6cm}|}
  			\hline $As_{2}^{4}\newline e_{1}.e_{1}=e_{1},e_{1}.e_{2}=e_{2},\newline e_{2}.e_{2}=e_{2}$ &  	$\begin{pmatrix}
  			1 & 0  \\
  			0 & a
  			\end{pmatrix},a\neq0$  &  $r_{\sharp}^{1}{=}\begin{pmatrix}
  			0 & 1  \\
  			1 & c
  			\end{pmatrix};r_{\sharp}^{2}{=}\begin{pmatrix}
  			0 & 0  \\
  			0 & 1
  			\end{pmatrix};\newline r_{\sharp}^{3}{=}\begin{pmatrix}
  			1 & 0  \\
  			0 & 0
  			\end{pmatrix}; r_{\sharp}^{4}{=}0$ \\
  			\hline 
  		\end{tabular}
  	\end{center}
  	
  	We end this paper by giving another proof to Lemma \ref{le1}.
  	
  	\begin{proof}
  		For any $\ga:[0,1]\too G\times G$, $t\mapsto (\ga_1(t),\ga_2(t))$ with $\ga(0)=(a,b)$ and $\ga(1)=(c,d)$,
  		\[ \tau_{m(\ga)}(T_{(a,b)}m(u,v))=T_{(c,d)}m(\tau_\ga(u,v)), \]where $\tau_\ga:T_{(a,b)}(G\times G)\too T_{(c,d)}(G\times G)$ and $\tau_{m\ga}:T_{ab}G\too T_{cd}G$ are the parallel transports.
  		But
  		\[ T_{(a,b)}m(u,v)=T_aR_b(u)+T_bL_a(v) \esp \tau_\ga(u,v)=(\tau_{\ga_1}(u),\tau_{\ga_2}(v)). \]So we get
  		\[ \tau_{\ga_1\ga_2}(T_aR_b(u))+\tau_{\ga_1\ga_2}(T_bL_a(v))=T_cR_d(\tau_{\ga_1}(u))+
  		T_dL_c(\tau_{\ga_2}(v)). \]
  		If we take $v=0$ and $\ga_2(t)=b=d$. We get
  		\[ \tau_{\ga_1b}(T_aR_b(u))=T_cR_b(\tau_{\ga_1}(u)) \]and hence $\na$ is right invariant. In the same way we get that $\na$ is left invariant. And finally
  		\[ \tau_{\ga_1\ga_2}(T_aR_b(u))=T_cR_d(\tau_{\ga_1}(u))\esp \tau_{\ga_1\ga_2}(T_bL_a(v))=T_dL_c(\tau_{\ga_2}(v)). \]If we take $\ga_2=\ga_1^{-1}$ we get that
  		\[ \tau_{\ga_1}(u)=T_aR_{a^{-1}c}(u)=T_aL_{ca^{-1}}(u). \]This implies that the adjoint representation is trivial and hence $G$ must be abelian.
  	\end{proof}
  	
\bibliographystyle{amsplain}

 \end{document}